\documentclass[12pt]{amsart}
\usepackage{amssymb, amscd}

\usepackage{mathptmx}

\usepackage[latin1]{inputenc}
\newcommand{\GG}{\mathfrak G}

\newcommand{\R}{\mathbb{R}}
\newcommand{\N}{\mathbb{N}}

\newcommand{\Z}{\mathbb{Z}}

\newcommand{\Px}{\mathbb{P}}

\newcommand{\C}{\mathbb{C}}
\newcommand{\K}{L}
\newcommand{\kx}{K}

\newcommand{\Cc}{\mathcal{C}}

\newcommand{\Fc}{\mathcal{F}}
\newcommand{\Gc}{\mathcal{G}}

\newcommand{\Wc}{\mathcal{W}}

\newcommand{\Dc}{\mathcal{D}}

\DeclareMathOperator{\pnt}{\raise 0.5mm \hbox{\large\bf.}}

\DeclareMathOperator{\GL}{GL}
\DeclareMathOperator{\inom}{in}
\DeclareMathOperator{\gT}{gT}

\DeclareMathOperator{\GF}{GF}

\DeclareMathOperator{\GC}{GC}
\DeclareMathOperator{\gGC}{gGC}

\DeclareMathOperator{\Gr}{Gr}

\newtheorem{theorem}{\bf Theorem} [section]
\newtheorem{lemma}[theorem]{\bf Lemma}
\newtheorem{cor}[theorem]{\bf Corollary}
\newtheorem{prop}[theorem]{\bf Proposition}

\theoremstyle{definition}
\newtheorem{defn}[theorem]{\bf Definition}

\newtheorem{notation}[theorem]{\bf Notation}
\newtheorem{rem}[theorem]{\bf Remark}
\newtheorem{ex}[theorem]{\bf Example}

\theoremstyle{plain}
\newtheorem*{satz*}{Theorem}

\title[GTV on Subvarieties and in the Non-constant Coefficient Case]{Generic Tropical Varieties on Subvarieties and in the Non-constant Coefficient Case}

\author{Kirsten Schmitz}
\address{Technische Universit\"at Kaiserslautern, Fachbereich Mathematik, 67653 Kaiserslau\-tern, Germany}
\email{schmitz@mathematik.uni-kl.de}

\begin{document}

\begin{abstract}
In earlier papers it was shown that the generic tropical variety of an ideal can contain information on algebraic invariants as for example the depth in a direct way. The existence of generic tropical varieties has so far been proved in the constant coefficient case for the usual notion of genericity. In this paper we generalize this existence result to include the case of non-constant coefficients in certain settings. Moreover, we extend the notion of genericity to arbitrary closed subvarieties of the general linear group. In addition to including the concept of genericity on algebraic groups this yields structural results on the tropical variety of an ideal under an arbitrary linear coordinate change.
\end{abstract}

\maketitle

\section{Introduction}

One aim of tropical algebraic geometry is to provide a tool to study certain algebraic varieties with the help of combinatorial objects associated to them, see for example \cite{DIFEST,DR,MI1,STTE}. These tropical varieties can be defined in various ways and settings, for instance as explained in \cite{GA}. We will use a definition relying on Gr\"obner basis theory as stated below.

A striking observation concerning tropical varieties as defined in this way is that they depend on the choice of coordinates of the polynomial ring containing the defining ideal. As algebraic invariants of its coordinate ring are, however, by definition independent of the choice of coordinates, the question arises whether there are generic tropical varieties which encode algebraic invariants in a direct way. In \cite{TK} it was shown that generic tropical varieties exist in the constant coefficient case for the usual notion of ``genericity''. These generic tropical varieties contain information on invariants as for example the depth and multiplicity of the coordinate ring in a direct way, see \cite{TK2}.

In this paper the existence result of \cite{TK} will be generalized in two ways. First it will include the non-constant coefficient case in the setting of the field of generalized power series $L$ over a given ground field $K$. The proofs for this can also be adapted to work for slightly different valued field, e.g the field of Puiseux series over $K$ (in which case we need to assume that the characteristic of $K$ is $0$) or the field introduced in \cite{MAR}, but rely on the structure of $L$ being a field of formal power series over $K$.

The second generalization is with respect to the notion of ``genericity''. In the classical statements on the existence of generic initial ideals in Gr\"obner basis theory (as in \cite[Section 15.9]{E} and \cite{GRE}) the term ``generic'' refers to the existence of a non-empty open subset $U$ of the general linear group $\GL_n(K)$ over $K$ such that a given property hold for all $g\in U$. Most proofs revolving around this notion, however, do not use any specific properties of $\GL_n(K)$ other than it being closed in itself and irreducible. We therefore use the notion of genericity with respect to any closed irreducible subvariety $V$ of $\GL_n(K)$ as explained in Section \ref{genericity} and prove the existence of a generic tropical variety on $V$. This leads to certain finiteness results regarding possible tropical varieties of arbitrary linear coordinate transformations, such as Corollary \ref{endlmoegl}.

This paper is organized as follows. In Section \ref{groebnercompl} the basic objects of study in our setting and our notation is introduced and the main result is summarized. Since the proofs of the main results depend on considering certain extensions of valued fields, the technical issues and statements concerning this are presented in Section \ref{valuedfields}. The precise definition of genericity used here are given in Section \ref{genericity} along with some general results needed in the following. In Section \ref{genericgroebnercomplexes} the existence of generic Gr\"obner complexes in this general setting is proved using the methods developed in \cite{MAST}. The proof of the existence of generic tropical varieties in this general meaning seems to be more involved. It relies on the ideas of \cite{HETH} where short tropical bases are produced with the help of rational projections. The technical generalizations concerning the auxiliary ideals and the rational projections introduced there to our context are explained in Sections \ref{techgensets} and \ref{ratproj}, respectively. In Section \ref{generictropicalvarieties} we give the proof of the existence of generic tropical varieties and generic tropical bases. Section \ref{examples} concludes the paper with some example classes for which generic Gr\"obner complexes and tropical varieties can be computed directly.

The material is to a large extent contained in \cite{ICH}.

\section{Preliminaries and Statement of Result}\label{groebnercompl}

For the following let $K$ be an algebraically closed field and $L$ be the field of generalized power series over $K$, see Section \ref{valuedfields} for the technical details on the valued fields needed for the proof of the main theorems. The assumption that $K$ be algebraically closed is needed for instance for Proposition \ref{extkrempel}. We will use the definition of Gr\"obner complexes and tropical varieties from \cite[Chapter 2]{MAST}. Following the notation there, for an element $a$ of the valuation ring $R_L$ of $L$ we denote by $\overline{a}$ the image of $a$ in the residue field of $R_L$ modulo its maximal ideal. Note that in our setting this field is canonically isomorphic to $K$ and we will identify it with $K$ in the following.

Let $S_{\K}={\K}[x_1,\ldots,x_n]$ and $S_K=K[x_1,\ldots,x_n]$ be the polynomial rings in $n$ variables over ${\K}$ and $K$, respectively. For a polynomial $f\in S_{\K}$ with $f=\sum_{\nu\in\N^n} a_{\nu} x^{\nu}$ all of whose coefficients $a_{\nu}$ are in $R_L$ we denote by $\overline{f}$ the polynomial $\sum_{\nu\in\N^n} \overline{a_{\nu}} x^{\nu}\in S_K$.

For $f\in S_L$ and $\omega\in\R^n$ we can consider the polynomial $f(x_1t^{\omega_1},\ldots,x_nt^{\omega_n})$, which is the image of $f$ under the ${\K}$-algebra automorphism on $S_{\K}$ induced by mapping $x_i$ to $x_it^{\omega_i}$. Let $$W=\min_{\nu}\left\{ v(a_{\nu})+\omega\cdot\nu\right\}.$$ Then $$\inom_{\omega}(f)=\overline{t^{-W}f(x_1t^{\omega_1},\ldots,x_nt^{\omega_n})}\in S_K$$ is called the
\emph{initial form of $f$ with respect to $\omega$}.

We consider graded ideals $I\subset S_{\K}$ with respect to the standard $\Z$-grading. For a graded ideal $I\subset S_{\K}$ the ideal $$\inom_{\omega}(I)=(\inom_{\omega}(f):f\in I)\subset S_K$$ is called the \emph{initial ideal of $I$ with respect to $\omega$}. We denote by $(S_{\K}/I)_d$ the ${\K}$-vector space of homogeneous elements of degree $d$ of $S_{\K}/I$ for $d\in \Z$.

In contrast to the classical setting the partial ordering of terms induced by their $\omega$-weights cannot be refined by a monomial ordering, since this partial order depends on the coefficients of the monomials appearing. This leads to some technical difficulties, see \cite{MAST}. However, the main properties of $I$ as a graded ideal in $S_L$ are preserved under the degeneration to $\inom_{\omega}(I)$. 

\begin{prop}[\cite{MAST}]\label{hilbinit}
Let $I\subset S_L={\K}[x_1,\ldots,x_n]$ be a graded ideal and $\omega\in\R^n$. Then $\inom_{\omega}(I)$ is a graded ideal and the Hilbert function of the two corresponding coordinate rings agree: For $d\geq 0$ we have $$\dim_{\K}(S_{\K}/I)_d=\dim_{\kx}(S_{\kx}/\inom_{\omega}(I))_d.$$ In particular, this implies equality for the Krull dimensions $\dim (S_{\K}/I)=\dim (S_{\kx}/\inom_{\omega}(I))$.
\end{prop}

In classical Gr\"obner basis theory for a graded ideal $I\subset S_K$ a complete fan in $\R^n$ is defined by the following equivalence relation, see \cite{MORO}. Two vectors $\omega,\omega'\in \R^n$ are equivalent if and only if $\inom_{\omega}(I)=\inom_{\omega'}(I)$. The equivalence class $C=C[\omega]$ of some $\omega\in \R^n$ is a relatively open cone and we denote by $\inom_{C}(I)$ the initial ideal corresponding to it. The topological closure of $C[\omega]$ is called a \emph{Gr\"obner cone} and the collection of all cones arising in this way form a complete fan, the \emph{Gr\"obner fan $\GF(I)$}, in $\R^n$.

In the non-constant coefficient case for graded ideals $I\subset S_L$ the set of all $\omega\in\R^n$ which induce the same ideal $\inom_{\omega}(I)$ are the relative interior of a polyhedron, called a \emph{Gr\"obner polyhedron}. The collection of all these polyhedra form a polyhedral complex in $\R^n$. All of this is proved in \cite[Chapter 2]{MAST}. For a graded ideal $I\subset S_L$ the polyhedral complex defined by all Gr\"ob\-ner polyhedra of $I$ is called the \emph{Gr\"obner complex $\GC(I)$} of $I$.

We now consider the zero-set $X(I)\subset {\K}^n$ of $I$ consisting of all $p\in {\K}^n$ with $f(p)=0$ for all $f\in I$. Note that it is not required that $I$ is a radical ideal. The notion of the tropical variety of $I$ originally describes the component-wise image of $X(I)$ under $v$, i.e. $$\left\{(v(p_1),\ldots,v(p_n)): p\in X(I)\right\}\cap \R^n.$$ For computational aspects there is a useful description of tropical varieties in terms of initial ideals, which is closely connected to the notion of initial ideals defined above. By the so called fundamental theorem of tropical geometry (see for example \cite[Theorem 4.2]{DR}) the tropical variety of a graded ideal $I\subset S_L$ as defined above can be identified with the set of all $\omega\in \R^n$, such that $\inom_{\omega}(f)$ is no monomial for all $f\in I$ or, equivalently, such that $\inom_{\omega}(I)$ contains no monomial. With this description the tropical variety is a subcomplex of the Gr\"obner complex of $I$ in a natural way. We consider this polyhedral complex structure as part of this notion and take this as our definition.

\begin{defn}
Let $I\subset S_L$ be a graded ideal. Then the subcomplex of the Gr\"obner complex of $I$ induced on the set $$\left\{\omega\in\R^n: \inom_{\omega}(I) \text{ does not contain a monomial } \right\}$$ will be called the tropical variety of $I$ and be denoted by $T(I)$.
\end{defn}

To be able to refer to it, we state the following theorem essentially proved in \cite[Theorem A]{BIGR}.

\begin{theorem}[\cite{BIGR}]\label{tropbasics}
Let $I\subset S_{\K}={\K}[x_1,\ldots,x_n]$ be a graded ideal. If we consider the tropical varieties as sets, we have:
\begin{enumerate}
\item $T(I)=T(\sqrt{I})=\bigcup_{P} T(P)$ where the union is taken over all minimal prime ideals $P$ of $I$.
\item If $I$ is prime with $\dim (S_{\K}/I)=m$ and $I$ does not contain a monomial, then $T(I)$ is a pure $m$-dimensional complex.
\item If $\dim (S_{\K}/I)=m$ and there exists a minimal prime $P$ of $I$ containing no monomial with $\dim (S_{\K}/P)=m$, then $\dim T(I)=m$.
\end{enumerate}
\end{theorem}

To compute tropical varieties the concept of a tropical basis is useful. Let $I\subset S_L$ be a graded ideal. Then a finite system of homogeneous generators $f_1,\ldots,f_t$ of $I$ is called a \emph{tropical basis} of $I$ if $$T(I)=\bigcap_{i=1}^t T(f_i).$$

In the constant coefficient case every ideal has a tropical basis as was observed in \cite[Theorem 2.9]{BOJESPSTTH}. (The proof of that paper also works for other fields than $\C$.) In the non-constant coefficient case tropical bases, which use a restricted number of polynomials, are constructed in \cite{HETH}. The methods of rational projections used there will be important for the proof of our main theorem.

The main result of this paper can be summarized as follows. Generalizing the result \cite[Theorem 1.1]{TK} we prove that for a graded ideal in $S_L$ a generic Gr\"obner complex and a generic tropical variety exist for the notion of genericity described in the introduction and elaborated in Section \ref{genericity}. More precisely, let $I\subset S_L$ be a graded ideal and $V\subset \GL_n(K)$ be a closed subvariety. Then there exists a Zariski-open set $\emptyset\neq U\subset V$ such that $\GC(g(I))$ and $T(g(I))$ are the same polyhedral complexes (respectively) for all $g\in U$. Moreover, there exists a notion of a generic tropical basis in the second case. This is stated and proved in Theorem \ref{gengrobcomp} for the Gr\"obner complex and in Theorem \ref{thm:main} for the tropical variety.

\section{Valued Fields}\label{valuedfields}

In this section the fields and valuations will be introduced, that are used in the following.

Let $K$ be a field endowed with the \emph{trivial valuation} $v$, where $v(0)=\infty$ and $v(a)=0$ for $a\in K^*$. This valuation gives rise to the so called \emph{constant coefficient case} in tropical geometry. To define tropical varieties in a meaningful way, however, field extensions of $K$ with a richer image are needed which inherit certain properties of $K$ depending on the setting. There are various possibilities to construct such field extensions. A prominent example for a valued field $(\K,v)$ extending $K$ with $v(\K)=\R\cup \left\{\infty\right\}$ is the \emph{field of generalized power series over $K$} which will be the construction used in the following. This is a special case of the following definition.

\begin{defn}
Let $R$ be a domain. The set $$R\left\{\!\!\left\{t\right\}\!\!\right\}=\left\{\sum_{\nu\in\R}c_{\nu}t^{\nu}: c_{\nu}\in R \text{ and } \left\{\nu: c_{\nu}\neq 0\right\} \text{ is well-ordered }\right\}$$ is called the \emph{ring of generalized power series over $R$}.
\end{defn}

Recall that this set with addition and multiplication analogously to those of polynomials is indeed a domain, see \cite[(1.14)]{RI}. If ${\kx}$ is an algebraically closed field, then so is ${\kx}\left\{\!\!\left\{t\right\}\!\!\right\}$, see \cite[(2.1) and (5.2)]{RI}. In this case there is a natural valuation on ${\kx}\left\{\!\!\left\{t\right\}\!\!\right\}$ defined by

\begin{eqnarray*}
v: {\kx}\left\{\!\!\left\{t\right\}\!\!\right\} & \longrightarrow & \R\cup\left\{\infty\right\}\\
f=\sum_{\nu\in\R}c_{\nu}t^{\nu}  & \longmapsto & \min\left\{\nu: c_{\nu}\neq 0\right\} \text{ if } f\neq 0
\end{eqnarray*}
and $v(0)=\infty$.  It is useful to view ${\kx}\left\{\!\!\left\{t\right\}\!\!\right\}$ as a valued field extension of ${\kx}$ endowed with the trivial valuation.

\begin{notation}\label{kandK}
In the following ${\kx}$ will always denote an algebraically closed field equipped with the trivial valuation and ${\K}={\kx}\left\{\!\!\left\{t\right\}\!\!\right\}$ the field of generalized power series over ${\kx}$ with the natural valuation as defined above.
\end{notation}

To obtain results on tropical varieties under generic coordinate transformations it is useful to first consider the coefficients of these coordinate changes as independent variables, see \cite{TK}. For a finite set of independent variables $Y$ and a field ${\kx}$ or $\K$ one can then do the necessary computations in ${\kx}(Y)$ or $\K(Y)$, respectively. Afterwards the desired coefficients for these variables can be substituted, see Section \ref{genericity} for the details. The main technical problem with this is that the field extension by $Y$ does not commute with taking the field of generalized power series, i.e. ${\kx}(Y)\left\{\!\!\left\{t\right\}\!\!\right\}$ is not canonically isomorphic to ${\kx}\left\{\!\!\left\{t\right\}\!\!\right\}(Y)$. The rest of this section is devoted to establishing a setting which copes with this difficulty.

Let $Y$ be a finite set of independent variables over ${\K}$ and consider the canonical inclusion of polynomial rings ${\kx}[Y] \hookrightarrow {\K}[Y]$. Moreover, for a domain $R$ let $Q(R)$ denote its quotient field. We need the following result:

\begin{prop}
Let ${\kx},{\K}$ and $Y$ be as defined above and $P\subset {\kx}[Y]$ be a prime ideal. Then there is a canonical inclusion of rings $${\K}[Y]/P{\K}[Y] \hookrightarrow ({\kx}[Y]/P)\left\{\!\!\left\{t\right\}\!\!\right\}.$$ In particular, $P{\K}[Y]$ is a prime ideal in ${\K}[Y]$ and we have a natural field extension $$Q({\K}[Y]/P{\K}[Y]) \hookrightarrow Q({\kx}[Y]/P)\left\{\!\!\left\{t\right\}\!\!\right\}.$$
\end{prop}

\begin{proof}
Note that every element $h\in {\K}[Y]$ is a formal sum $\sum_{\nu} (\sum_{\mu\in \R} a_{\nu\mu} t^{\mu})y^{\nu}$, where we abbreviate $y_1^{\nu_1}\cdots y_m^{\nu_m}$ by $y^\nu$ for $\nu\in \N^m$, the set of $\nu$ appearing as exponents is finite and for every $\nu$ the set $\left\{\mu: a_{\nu\mu}\neq 0\right\}$ is well-ordered.

We define the ring homomorphism
\begin{eqnarray*}
\psi: {\K}[Y] & \longrightarrow & ({\kx}[Y]/P)\left\{\!\!\left\{t\right\}\!\!\right\}\\
\sum_{\nu} (\sum_{\mu\in \R} a_{\nu\mu} t^{\mu})y^{\nu} & \longmapsto & \sum_{\mu\in \R} (\sum_{\nu} a_{\nu\mu}y^{\nu}+P) t^{\mu},
\end{eqnarray*}
which is well-defined, since for a given $\mu$ there exist only finitely many $\nu$, such that $a_{\nu\mu}\neq 0$. We show that $\ker \psi =P{\K}[Y]$. First note that for $p\in P\subset {\kx}[Y]\hookrightarrow {\K}[Y]$ we have $\psi(p)=0$ by definition. Since $\psi$ is a ring homomorphism, this implies $$\psi(\sum p_ih_i)=\sum \psi(p_i)\psi(h_i)=0$$ for every finite sum with $p_i\in P$ and $h_i\in {\K}[Y]$. Hence, $P{\K}[Y]\subset \ker \psi$.

For the other inclusion let $h=\sum_{\nu} (\sum_{\mu\in \R} a_{\nu\mu} t^{\mu})y^{\nu}\in \ker \psi$. Then $$\sum_{\mu\in \R} (\sum_{\nu} a_{\nu\mu}y^{\nu}+P) t^{\mu}=0,$$ so $\sum_{\nu}a_{\nu\mu}y^{\nu}\in P$ for every $\mu$ appearing. Choose an exponent $\nu_0\in\N^m$ with non-zero coefficient $\sum_{\mu\in\R} a_{\nu_0\mu}t^{\mu}$ in $h$. Furthermore, choose $\mu_0$, such that $a_{\nu_0\mu_0}\neq 0$. Since $$p_1:=\frac{1}{a_{\nu_0\mu_0}}\sum_{\nu} a_{\nu\mu_0}y^{\nu}\in P,$$ we can write $y^{\nu_0}=\sum_{\nu\neq\nu_0} a'_{\nu\mu_0}y^{\nu}+p_1$, where $a'_{\nu\mu_0}=-a_{\nu\mu_0}/a_{\nu_0\mu_0}$. Hence, 
\begin{eqnarray*}
h & = & (\sum_{\mu\in \R} a_{\nu_0\mu}t^{\mu})y^{\nu_0}+ \sum_{\nu\neq \nu_0} (\sum_{\mu\in \R} a_{\nu\mu} t^{\mu})y^{\nu}\\
  & = & (\sum_{\mu\in \R} a_{\nu_0\mu}t^{\mu})(\sum_{\nu\neq\nu_0} a'_{\nu\mu_0}y^{\nu}+p_1)+\sum_{\nu\neq \nu_0} (\sum_{\mu\in \R} a_{\nu\mu} t^{\mu})y^{\nu}\\
  & = & p_1 h_1 + \sum_{\nu\neq\nu_0} (\sum_{\mu\in\R} b_{\nu\mu}t^{\mu})y^{\nu},
\end{eqnarray*}
where $p_1\in P$, $h_1\in {\K}[Y]$ and the right part is a polynomial in ${\K}[Y]$ containing one less term than $h$. By induction on the number of terms of $h$ we obtain a finite expression $h=\sum p_ih_i$ with $p_i\in P$, $h_i\in {\K}[Y]$, so $h\in P{\K}[Y]$. This shows that $P{\K}[Y]= \ker \psi$. The map $\psi$ therefore induces a canonical inclusion $${\K}[Y]/P{\K}[Y] \hookrightarrow ({\kx}[Y]/P)\left\{\!\!\left\{t\right\}\!\!\right\}.$$ In particular, $P{\K}[Y]$ is a prime ideal, as $({\kx}[Y]/P)\left\{\!\!\left\{t\right\}\!\!\right\}$ is a domain. Moreover, since $Q({\kx}[Y]/P)\left\{\!\!\left\{t\right\}\!\!\right\}$ is a field, this inclusion induces the desired field extension $$Q({\K}[Y]/P{\K}[Y]) \hookrightarrow Q({\kx}[Y]/P)\left\{\!\!\left\{t\right\}\!\!\right\}.$$
\end{proof}

The field $Q({\K}[Y]/P{\K}[Y])$ will play a fundamental role in the following sections, since it provides the right tool to deal with ``genericity'' on the subvariety $V\subset \GL_n({\kx})$ which is the zero-set of $P$, see Section \ref{genericity} for this notion.

With the above result we obtain a natural valuation on $Q({\K}[Y]/P{\K}[Y])$ which extends the valuation on ${\K}$.

\begin{cor}\label{valfieldext}
The chain of inclusions $${\kx}\left\{\!\!\left\{t\right\}\!\!\right\}={\K}\hookrightarrow Q({\K}[Y]/P{\K}[Y]) \hookrightarrow Q({\kx}[Y]/P)\left\{\!\!\left\{t\right\}\!\!\right\}$$ is an inclusion of valued fields, where the valuations on ${\K}$ and $Q({\kx}[Y]/P)\left\{\!\!\left\{t\right\}\!\!\right\}$ are the natural valuations as fields of generalized power series and the valuation on $Q({\K}[Y]/P{\K}[Y])$ is the restriction of the one on $Q({\kx}[Y]/P)\left\{\!\!\left\{t\right\}\!\!\right\}$.
\end{cor}

We will use the following notation.

\begin{notation}\label{kv}
Let ${\kx}\hookrightarrow {\K}$ be as in Notation \ref{kandK} and $V\subset \GL_n({\kx})$ be a subvariety defined by a prime ideal $P$. In the following ${\K}(V)$ will always denote the field $Q({\K}[Y]/P{\K}[Y])$ as constructed above with the valuation of Corollary \ref{valfieldext}. In addition ${\kx}(V)$ will denote the quotient field $Q({\kx}[Y]/P)$ equipped with the trivial valuation.
\end{notation}

\section{Genericity}\label{genericity}

As introduced in Section \ref{valuedfields} we consider the field extension ${\kx}\hookrightarrow {\K}$ of valued fields, where ${\kx}$ is an algebraically closed field equipped with the trivial valuation and ${\K}={\kx}\left\{\!\!\left\{t\right\}\!\!\right\}$ is the field of generalized power series over ${\kx}$.

In this section we will specify the meaning of the term \emph{generic} for this note and introduce the notation used here. This notion of genericity differs from the one used in \cite{E,GRE} and also so in \cite{TK,TK2} in two ways. First of all we will not consider arbitrary coordinate transformations. Since we are dealing with valued fields and the valuation has a great influence on taking initial ideals as introduced in Section \ref{groebnercompl}, coordinate transformations involving field elements of non-zero valuation will not yield any ``generic'' results, see \cite[Remark 2.8]{TK}. We therefore only consider coordinate transformations by elements of $\GL_n({\kx})$ instead of the whole general linear group $\GL_n({\K})$.

Moreover, we will generalize the meaning of ``generic'' to arbitrary irreducible subvarieties of $\GL_n({\kx})$. We will consider $\GL_n({\kx})$ as an affine ${\kx}$-space of dimension $n^2$ equipped with the Zariski topology. In the classical setting in Gr\"obner basis theory the term generic is used, if there exists a non-empty Zariski-open subset $U\subset \GL_n({\kx})$, such that all $g\in U$ fulfill a given condition. Such a set $U$ is by definition of the Zariski topology dense in $\GL_n({\kx})$, so the name ``generic'' is justified. By a subvariety of $V\subset\GL_n({\kx})$ we will mean a non-empty irreducible closed subset.  As we would like to deal with properties of subvarieties of $V$ as well, we will use the notion ``generic for $V$'', meaning there is a non-empty Zariski-open subset of $V$ (in the induced topology) satisfying the given condition. In particular, this allows us to extend our results to algebraic subgroups of $\GL_n({\kx})$ as well, see Section \ref{examples} for a discussion of the subgroup of diagonal matrices. In addition this concept can yield results on the number different outcomes for all coordinate transformations in $\GL_n({\kx})$, see Corollary \ref{endlmoegl}.

To handle generic coordinate transformations the following ${\K}$-algebra homomorphism plays a fundamental role.

\begin{defn}\label{gendef}
Let $Y=\left\{y_{ij}: i,j=1,\ldots,n\right\}$ be a set of $n^2$ independent variables over ${\kx}$ and $V\subset \GL_n({\kx})$ be a subvariety defined by the prime ideal $P\subset {\kx}[Y]$. Let ${\K}(V)$ be the field extension of ${\K}$ from Notation \ref{kv}. In the following we consider the ${\K}$-algebra homomorphism induced by

\begin{eqnarray*}
y: {\K}[x_1,\ldots,x_n] & \longrightarrow & {\K}(V)[x_1,\ldots, x_n]\\
                 x_i & \longmapsto     & \sum_{j=1}^n (y_{ji}+P{\K}[Y]) x_j.
\end{eqnarray*}

For any $g=(g_{ij})\in V\subset \GL_n({\kx})$ this induces an ${\K}$-algebra automorphism on
${\K}[x_1,\ldots,x_n]$ by substituting $g_{ij}$ for $y_{ij}$. We identify $g$ with the induced
automorphism and use the notation $g$ for both of them.
\end{defn}

In addition we will sometimes use the restricted ${\kx}$-algebra homomorphism induced by

\begin{eqnarray*}
y: {\kx}[x_1,\ldots,x_n] & \longrightarrow & {\kx}(V)[x_1,\ldots, x_n]\\
                 x_i & \longmapsto     & \sum_{j=1}^n (y_{ji}+P) x_j.
\end{eqnarray*}

Note that for any $g\in\GL_n({\kx})$ the ideal $g(I)$ is a graded ideal isomorphic to $I$ as a graded ${\K}$-module. In particular, $g(I)$ has the same Hilbert function as $I$. On the other hand the set $y(I)\subset {\K}(V)[x_1,\ldots, x_n]$ is not an ideal in general. In this case we will be interested in the ideal generated by $y(I)$ in ${\K}(V)[x_1,\ldots, x_n]$ and by abuse of notation denote this ideal by $y(I)$. Moreover, we will sometimes denote a polynomial in ${\K}(V)[x_1,\ldots, x_n]$ in the form $f(y)$ to emphasize its dependence on the $y_{ij}$. Analogously, we denote the polynomial obtained by substituting $g\in\GL_n(K)$ for $y$ (if this is possible, i.e. if no denominator of the coefficients in the $y_{ij}$ vanishes) by $f(g)$.

In the situation that ${\K}$ is algebraically closed, this extension of ideals preserves the main structural features of ideals which are important to us. This is due to the following proposition.

\begin{prop}\label{extkrempel}
Let ${\K}\subset {\K}'$ be a field extension and consider an $S_L$-algebra inclusion $$S_{\K}={\K}[x_1,\ldots,x_n]\hookrightarrow S_{{\K}'}={\K}'[x_1,\ldots,x_n].$$ For any prime ideal $P\subset S_{\K}$ the extension $PS_{{\K}'}$ is also prime. If $P_1,\ldots,P_s$ are the minimal primes of an arbitrary graded ideal $I\subset S_{\K}$, then $P_1S_{{\K}'},\ldots,P_sS_{{\K}'}$ are the minimal primes of $IS_{{\K}'}$. Moreover, for each homogeneous component of $I$ we have $\dim_{\K} I_d=\dim_{{\K}'} (IS_{{\K}'})_d$. This implies that the Hilbert functions and, hence, the Krull dimensions of $I$ and $IS_{{\K}'}$ coincide.
\end{prop}

\begin{proof}
See \cite[Chapter II, Exercise 3.15]{H} for the first statement, which can be applied, since ${\K}$ is algebraically closed.

The second statement follows from the fact that ``going down'' holds for flat extensions: For a prime ideal $P\subset S_{\K}$ we first show that $PS_{{\K}'}\cap S_{\K}=P$. The inclusion $PS_{{\K}'}\cap S_{\K}\supset P$ is clear. Since $S_{\K}\hookrightarrow S_{{\K}'}$ is flat, by ``going down'' (see \cite[Lemma 10.11]{E}) we have $Q\cap S_{\K}=P$ for any minimal prime $Q\subset S_{{\K}'}$ over $PS_{{\K}'}$. But since $PS_{{\K}'}$ is itself prime by the first statement, this implies $PS_{{\K}'}\cap S_{\K}=P$.

From this it follows directly that the minimal primes of $I$ and $IS_{{\K}'}$ correspond to each other: Let $P$ be a minimal prime of $I$. Then $IS_{{\K}'}\subset PS_{{\K}'}$ and $PS_{{\K}'}$ is prime. Assume that there is a prime ideal $Q\subset S_{{\K}'}$ with $IS_{{\K}'}\subset Q\subset PS_{{\K}'}$. By contracting and the above result we have $I\subset Q\cap S_{\K}\subset P$ and $P$ is minimal over $I$, so $Q\cap S_{\K}=P$. Thus $PS_{{\K}'}=(Q\cap S_{\K})S_{{\K}'}\subset Q$, which implies $Q= PS_{{\K}'}$. Hence, all the ideals $P_1S_{{\K}'},\ldots,P_sS_{{\K}'}$ are minimal primes of $IS_{{\K}'}$. To show that there can be no other minimal primes let $Q$ be any minimal prime of $IS_{{\K}'}$. Then $I\subset (Q\cap S_{\K})$, the latter of which is prime. Since $P_1,\ldots,P_s$ are the minimal primes of $I$, there exists an index $l$, such that $I\subset P_l\subset (Q\cap S_{\K})$. So $$IS_{{\K}'}\subset P_lS_{{\K}'}\subset (Q\cap S_{\K})S_{{\K}'}\subset Q.$$ Thus $P_lS_{{\K}'}=Q$, since both are minimal primes. This proves the second claim.

To prove the last claim let $h_1,\ldots,h_D$ be an ${\K}$-vector space basis of $I_d$, so $I_d=\bigoplus_{i=1}^D {\K}\cdot h_i$. Since $(IS_{{\K}'})_d=I_d\otimes_{\K} {\K}'$ as an ${\K}'$-vector space, we have $$(IS_{{\K}'})_d=I_d \otimes_{\K} {\K}'=(\bigoplus_{i=1}^D {\K}\cdot h_i)\otimes_{\K} {\K}'=\bigoplus_{i=1}^D ({\K} \otimes_{\K} {\K}')\cdot h_i=\bigoplus_{i=1}^D {\K}'\cdot h_i,$$ as the tensor product commutes with direct sums. Hence, $(IS_{{\K}'})_d$ is $D$-dimensional as an ${\K}$-vector space proving the claim.
\end{proof}

Note that all statements of Proposition \ref{extkrempel} apply to the ideal $y(I)$ generated by the image of $I$ under $y$:

\begin{rem}\label{primesinext}
Proposition \ref{extkrempel} implies that for a prime ideal $P\subset S_{\K}$ the ideal $y(P)\subset S_{\K}(V)$ is also prime. Moreover, for an arbitrary ideal $I\subset S_{\K}$ with minimal primes $P_1,\ldots,P_s$ the extension $y(I)$ of $I$ under the ${\K}$-algebra homomorphism $y$ from Definition \ref{gendef} has the minimal primes $y(P_1),\ldots,y(P_s)$ and the same Krull dimension as $I$.
\end{rem}

The concept of genericity as defined above will now be applied to introduce the generic objects used in the following. As we have extended the meaning of ``generic'' to subvarieties of $\GL_n({\kx})$, the questions of \cite{TK} can be adapted to ask for the existence of a generic Gr\"obner complex and a generic tropical variety of $I$ on a subvariety $V\subset \GL_n({\kx})$. 

\begin{defn}\label{tropdef}
Let $V\subset \GL_n({\kx})$ be a subvariety and $I\subset S_{\K}$ be a graded ideal.
\begin{enumerate}
\item If for an open subset $\emptyset \neq U\subset V$ the Gr\"obner complex $\GC(g(I))$ is the same polyhedral complex for all $g\in U$, then this complex is called the \emph{generic Gr\"obner complex of $I$ on $V$}. It will be denoted by $\gGC_V(I)$.
\item If $T(g(I))$ is the same complex for all $g$ in an open subset $\emptyset\neq U\subset V$, then this complex is called the \emph{generic tropical variety of $I$ on $V$} and is denoted by $\gT_V(I)$.
\end{enumerate}
If $V$ is clear from the context we will also denote $\gGC_V(I)$ by $\gGC(I)$ and $\gT_V(I)$ by $\gT(I)$, respectively.
\end{defn}

A priori it is of course not clear, that generic Gr\"obner complexes or generic tropical varieties exist. The proof of this will be the object of Section \ref{genericgroebnercomplexes} and Section \ref{generictropicalvarieties}, respectively.

Note, however, that in the constant coefficient case the existence of a generic Gr\"obner fan and generic universal Gr\"obner basis on a subvariety $V$ of $\GL_n(K)$ can be proved with the same method as in the proof of \cite[Theorem 3.1]{TK}, where the field $$K'=K(y_{ij}: i,j=1,\ldots,n)$$ is replaced by $K(V)$. This yields the following theorem also needed in a later proof.

\begin{theorem}\label{gengroebfanonsubvar}
Let $I\subset S_K$ be a graded ideal and $V\subset \GL_n(K)$ a subvariety. Then there exists an open set $\emptyset\neq U\subset V$ and polynomials $h_1(y),\ldots,h_s(y)\in y(I)\subset S_{K(V)}$ such that
\begin{enumerate}
\item $\left\{h_1(y),\ldots,h_s(y)\right\}$ is a universal Gr\"obner basis of $y(I)$.
\item For $g\in U$ the set $\left\{h_1(g),\ldots,h_s(g)\right\}$ is a universal Gr\"obner basis of $g(I)$.
\item All of these Gr\"obner bases have the same support.
\end{enumerate}
\end{theorem}

As another first result in this direction we note that generically the tropical variety of an ideal is empty if and only if $\dim (S_{\K}/I)=0$, the proof of which works exactly as the one of the analogous statement in \cite[Lemma 2.5]{TK} for the constant coefficient case.

\begin{prop}\label{dimo}
Let $I\subset S_{\K}={\K}[x_1,\ldots,x_n]$ be a graded ideal with $\dim (S_{\K}/I)>0$. Then there exists an open subset $\emptyset\neq U\subset \GL_n({\kx})$, such that $T(g(I))\neq \emptyset$ for every $g\in U$.
\end{prop}

Note that every graded ideal $I\subset S_L$ with $\dim (S_{\K}/I)=0$ contains a monomial. Thus Proposition \ref{dimo} immediately implies that if it exists, $\gT(I)=\emptyset$ if and only if $\dim (S_{\K}/I)=0$.

Recall that $\dim T(I)$ can be strictly smaller than $\dim (S_L/I)$ if $I$ is not prime and the minimal primes of $I$ defining the dimension contain monomials (this follows from Theorem \ref{tropbasics}). The above proposition shows that in general, however, equality holds between the dimensions even in the case of non-prime ideals.

\begin{cor}\label{genericsamedim}
Let $I\subset S_{\K}={\K}[x_1,\ldots,x_n]$ be a graded ideal. There exists an open subset $\emptyset\neq U\subset \GL_n({\kx})$, such that $\dim T(g(I))=\dim (S_{\K}/I)$ for every $g\in U$.
\end{cor}

\begin{proof}
The case $\dim (S_{\K}/I)=0$ is clear. Let $\dim (S_{\K}/I)=m>0$. Then there exists a minimal prime $P$ of $I$ with $\dim (S_{\K}/P)=m$. By Proposition \ref{dimo} there exists an open subset $\emptyset\neq U\subset \GL_n({\kx})$, such that $T(g(P))\neq \emptyset$ for all $g\in U$. Since $g(P)$ does not contain a monomial for $g\in U$, Theorem \ref{tropbasics} implies that $\dim T(g(I))=m$ for $g\in U$.
\end{proof}

If the generic tropical variety exists for an ideal $I$ on a subvariety $V$ of $\GL_n({\kx})$, we can also hope to find a tropical basis of each ideal $g(I)$, such that each member cuts out the same tropical hypersurface generically. This concept is encoded in the following definition.

\begin{defn}\label{gentropbasis}
Let $I\subset {\K}[x_1,\ldots,x_n]$ a graded ideal and $V\subset \GL_n({\kx})$ a subvariety. Let $y(I)\subset {\K}(V)[x_1,\ldots,x_n]$ be as in Definition \ref{gendef}. A finite set of polynomials $$F_1(y),\ldots,F_s(y)\in y(I)\subset {\K}(V)[x_1,\ldots,x_n]$$ is called a \emph{generic tropical basis of $I$ on $V$}, if there exists an open set $\emptyset\neq U\subset V$, such that:
\begin{enumerate}
\item $F_1(g),\ldots,F_s(g)$ is a tropical basis of $g(I)$ for $g\in U$.
\item For every $j$ we have: The tropical variety $T(F_j(g))$ is the same polyhedral complex for every $g\in U$.
\end{enumerate}
If $\emptyset\neq U\subset V$ fulfills these two conditions, the generic tropical basis is called \emph{valid} on $U$.
\end{defn}

The existence of generic tropical bases will be shown in Section \ref{generictropicalvarieties}.

\section{Generic Gr\"obner Complexes}\label{genericgroebnercomplexes}

Let ${\kx}$ and ${\K}$ be as defined in Notation \ref{kandK}. In \cite[Corollary 3.2]{TK} the existence of a generic Gr\"obner fan of a graded ideal $I\subset {\kx}[x_1,\ldots,x_n]$ was proved. In the setting of this paper this is the same as showing that a generic Gr\"obner complex of a graded ideal exists in the constant coefficient case, that is if the ideal $I\subset {\K}[x_1,\ldots,x_n]$ is generated in ${\kx}[x_1,\ldots,x_n]$, see \cite[Chapter 2]{MAST} in a section on Gr\"obner bases.

In the non-constant coefficient case a similar result can be proved. The proof given here relies on the fact that ${\K}$ is a field of power series over ${\kx}$ and a priori does not apply in a more general setting. The idea of the proof is taken from \cite{MAST}, where the concept of Gr\"obner complexes is introduced and their existence is shown. Since only graded ideals are considered, one can prove certain claims by considering the homogeneous components of the ideals separately. These are finitely generated vector spaces, which can be compared by studying the corresponding Grassmannians embedded into projective space via the Pl\"ucker embedding. 

Let $U$ be a $D$-dimensional vector subspace of an $N$-dimensional $L$-vector space. By choosing a basis we represent $U$ as the row space of a $D\times N$-matrix $A$ with entries in ${\K}$. We set $m={N\choose D}$ and consider the vector $P\in {\K}^{m}$ of all $D\times D$-minors of $A$. The components of $P$ will be indexed by subsets of the $N$ columns of $A$ of cardinality $D$. Following the notation in \cite{MAST} in this section we denote the set of all such subsets by $[N]^D$.

The equivalence class (up to scalar multiple) of $P$ in projective space $\Px^{m-1}$ is called the Pl\"ucker coordinates of $U$ in the Grassmannian $\Gr_{\K}(D,N)$. Note that the components $P_J$ of $P$ for $J\in [N]^D$ are elements of ${\K}$, so we can consider the componentwise valuation $v(P)$. This is not a well-defined concept on the Pl\"ucker coordinates, since these are defined up to ${\K}$-scalar multiple. In the following we will always mean that we apply the valuation map to a fixed representative, when we consider the valuation of Pl\"ucker coordinates. The final results will always depend on differences $v(P_J)-v(P_J')$ for $J,J'\in [N]^D$ and this is well-defined on the equivalence classes induced by scalar multiplication.

The aim of this section will be to prove the following theorem, which is an analogue to \cite[Theorem 3.1]{TK} in the non-constant coefficient case.

\begin{theorem}\label{gengrobcomp}
Let $I\subset S_{\K}$ be a graded ideal and $V\subset \GL_n({\kx})$ be a subvariety. Then there exists an open set $\emptyset\neq U\subset V$, such that the Gr\"obner complex $\GC(g(I))$ is the same complex for all $g\in U$.
\end{theorem}

This theorem shows that the first part of Definition \ref{tropdef} is not vacuous, since generic Gr\"obner complexes indeed exist. To prove this theorem we consider the graded components of the initial ideals of $g(I)$ for $g\in V$. These graded components each induce a polyhedral complex in $\R^n$, see \cite{MAST}. Let $d\geq 0$. For $g\in V$ and $\omega\in \R^n$ set $$C_{g(I)}^d[\omega]=\left\{\omega'\in\R^n:\inom_{\omega'}(g(I))_d=\inom_{\omega}(g(I))_d\right\}.$$ We call the topological closure of this the \emph{Gr\"obner polyhedron of $\omega$ in degree $d$}. The name is justified by the following statement.

\begin{lemma}\label{groebpoly}
Let $I$ and $V$ be as in Theorem \ref{gengrobcomp} and $d\geq 0$. Then there exists an open set $\emptyset\neq U(d)\subset V$, such that for every $\omega\in \R^n$ the set $$C_{g(I)}^d[\omega]=\left\{\omega'\in\R^n:\inom_{\omega'}(g(I))_d=\inom_{\omega}(g(I))_d\right\}$$ is the same relatively open polyhedron for all $g\in U(d)$.
\end{lemma}

\begin{proof}
We follow the proof of the corresponding statement in \cite{MAST} using almost the same notation and making the necessary observations for our result.

Consider the ${\K}$-algebra homomorphism $y:{\K}[x_1,\ldots,x_n] \longrightarrow {\K}(V)[x_1,\ldots,x_n]$ from Definition \ref{gendef}. Recall that $y(I)\subset {\K}(V)[x_1,\ldots,x_n]$ denotes the graded ideal generated by the image of $I$ under $y$.

Let $\dim_{\K} I_d=D$. We then have $\dim_{{\K}(V)} y(I)_d=D$ and $\dim_{\K} g(I)_d=D$ for all $g\in \GL_n({\kx})$ by Remark \ref{primesinext}. Moreover, by Proposition \ref{hilbinit} we also know that $\dim_{\kx} \inom_{\omega}(g(I))_d=D$ for $g\in\GL_n({\kx})$. Set $N={{n+d-1}\choose{d}}$ and enumerate all monomials of degree $d$ by $x^{\mu_1},\ldots,x^{\mu_N}$. The ${\K}(V)$-vector space $y(I)_d$ corresponds to a point in the Grassmannian $\Gr_{{\K}(V)}(D,N)$ and for all $g\in \GL_n({\kx})$ the ${\K}$-vector spaces $g(I)_d$ correspond to points in the Grassmannian $\Gr_{\K}(D,N)$. Moreover, the ${\kx}$-vector spaces $\inom_{\omega}(g(I))_d$ correspond to points in the Grassmannian $\Gr_{\kx}(D,N)$. 

Let $h_1(y),\ldots,h_D(y)$ be an ${\K}(V)$-basis of $y(I)_d$. Note that by multiplying with denominators we can choose the coefficients of the terms as polynomials in the residue classes of the $y_{ij}$ modulo the defining prime ideal of $V\subset \GL_n({\kx})$. This implies that the components $P_J(y)$ of the Pl\"ucker coordinates of $y(I)_d$ for $J\in [N]^D$ are also polynomials in the residue classes of the $y_{ij}$. We claim that there exists an open set $\emptyset\neq U\subset V$, such that $v(P_J(g))=v(P_J(g'))$ for all $g,g'\in U$ and all $J\in [N]^D$. 

To prove this consider $P_J(y)$ as an element of ${\kx}(V)\left\{\!\!\left\{t\right\}\!\!\right\}$ by the natural inclusion $${\K}(V)\hookrightarrow {\kx}(V)\left\{\!\!\left\{t\right\}\!\!\right\}$$ as in Corollary \ref{valfieldext}. Thus we write $P_J(y)$ as a formal power series in $t$ whose coefficients are polynomial expressions in the residue classes of the $y_{ij}$. Choose an open subset $\emptyset\neq U\subset V$, such that no leading coefficient in any of the $P_J(y)$ as an element of the valued field ${\kx}(V)\left\{\!\!\left\{t\right\}\!\!\right\}$ vanishes. This implies $v(P_J(g))=v(P_J(g'))$ for all $g,g'\in U$ and every $J\in[N]^D$.

In particular, $P_J(y)=0$ if and only if $P_J(g)=0$ for $g\in U$. So $h_1(g),\ldots,h_D(g)$ is a basis of the ${\K}$-vector space $g(I)_d$, since these vectors are linearly independent and the dimensions of $y(I)_d$ and $g(I)_d$ coincide. Hence, $P_J(g)$ are the Pl\"ucker coordinates of $g(I)_d$ for $g\in U$.
 
For $J\in [N]^D$ and $g\in U$ let $M_J=\sum_{j\in J}\mu_j$ and $W(\omega)=\min_{J}\left\{v(P_J(g))+\omega\cdot M_J\right\}$. Denote by $p^{\omega}_J(g)$ the Pl\"ucker coordinates of $\inom_{\omega}(g(I))_d$ depending on $\omega\in\R^n$.

As proved in \cite{MAST} the equation 
\begin{equation*}
p^{\omega}_J(g)=\overline{t^{\omega\cdot M_J-W(\omega)}P_J(g)}
\end{equation*}
holds for $g\in U$ up to global scaling, which does not change the point in the Grassmannian. Thus for $\omega'\in \R^n$ and $g\in U$ we have 
\begin{eqnarray*}
\omega'\in C_{g(I)}^d[\omega] & \Leftrightarrow & \inom_{\omega'}(g(I))_d=\inom_{\omega}(g(I))_d\\
                              & \Leftrightarrow & p_J^{\omega'}(g)=p_J^{\omega}(g) \hspace{0.3cm} \forall J\in [N]^D\\
                              & \Leftrightarrow & \overline{t^{\omega'\cdot M_J-W(\omega')}P_J(g)}=\overline{t^{\omega\cdot M_J-W(\omega)}P_J(g)} \hspace{0.3cm} \forall J\in [N]^D.
\end{eqnarray*}

If for some $J\in [N]^D$ we have $P_J(g)=0$, then also $p_J^{\omega}(g)=0$ for every $\omega$, so the above statement does not impose a condition on $C_{g(I)}^d[\omega]$. If $P_J(g)\neq0$, the last equation is fulfilled if for every set $J$ we have:
\begin{enumerate}
\item Either the minimum $W(\omega)$ is not attained at $J$, so $v(P_J(g))+\omega\cdot M_J>W(\omega)$. Then $p_J^{\omega'}(g)=p_J^{\omega}(g)=0$ and $v(P_J(g))+\omega'\cdot M_J>W(\omega')$ as well.
\item Or we have $v(P_J(g))+\omega\cdot M_J=W(\omega)$, then the Pl\"ucker coordinates coincide if and only if $v(P_J(g))+\omega'\cdot M_J=W(\omega')$.
\end{enumerate}
 These equalities and inequalities define $C_{g(I)}^d[\omega]$ to be the relative interior of a polyhedron in $\R^n$, which does not depend on $g$ for $g\in U$, as $v(P_J(g))$ is the same for every $g\in U$.
\end{proof}

After having obtained individual ``generic'' Gr\"obner polyhedra in a given degree in Lemma \ref{groebpoly}, these can now be shown to form a polyhedral complex in $\R^n$. This has been proved in \cite{MAST}. 

\begin{lemma}[\cite{MAST}]\label{degcomplex}
For all $g\in U(d)$ as in Lemma \ref{groebpoly} the collection of the closures of all polyhedra $C_{g(I)}^d[\omega]$ for $\omega\in\R^n$ form the same polyhedral complex $\Cc^d$ in $\R^n$.
\end{lemma}

With these prerequisites the proof of Theorem \ref{gengrobcomp} can be completed in the same way as is done in \cite{MAST}.

\begin{proof}[Proof of Theorem \ref{gengrobcomp}]
For each $d\in \N$ there exists an open set $\emptyset\neq U(d)\subset V$, such that the collection of the closures of all $C_{g(I)}^d[\omega]$ for $\omega\in\R^n$ is a fixed polyhedral complex $\Cc^d$ in $\R^n$ for $g\in U(d)$ by Lemma \ref{degcomplex}. It remains to show that there is a finite set $\Dc\subset \N$ and an open set $\emptyset\neq U\subset V$, such that for $g\in U$ we have $$\inom_{\omega}(g(I))=\inom_{\omega'}(g(I)) \Leftrightarrow \inom_{\omega}(g(I))_d=\inom_{\omega'}(g(I))_d \text{ for all }d\in \Dc.$$ In this case we consider the common refinement of all $\Cc^d$ for $d\in \Dc$ containing the closures of the relatively open polyhedra $C_{g(I)}[\omega]=\bigcap_{d\in\Dc} C_{g(I)}^d[\omega]$. These polyhedra are the equivalence classes of the relation of inducing the same initial ideal $\inom_{\omega}(g(I))=\inom_{\omega'}(g(I))$ for two elements $\omega,\omega'\in \R^n$ for $g\in U$. This proves Theorem \ref{gengrobcomp}.

To prove the above claim recall that $g(I)\subset {\K}[x_1,\ldots,x_n]$ has the same Hilbert function for every $g\in V$. Moreover, the Hilbert function is preserved by taking initial ideals $\inom_{\omega}(g(I))\subset {\kx}[x_1,\ldots,x_n]$ by Proposition \ref{hilbinit}. The Hilbert function is also preserved if initial ideals of $\inom_{\omega}(g(I))$ are taken in classical Gr\"obner basis theory with respect to some term order, (see \cite[Theorem 15.26]{E}). Any such initial ideal is one of finitely many monomial ideals, as there are only finitely many monomial ideals with the same Hilbert function, see \cite[Corollary 2.2]{MA}. We can now take $\Dc$ to be the set of total degrees of all minimal generators of all these monomial ideals, since every ideal $\inom_{\omega}(g(I))$ has a Gr\"obner basis of polynomials in these degrees. The claim now follows from the general fact that two graded ideals coincide, if they coincide in the degrees appearing in a generating system for each of them.
\end{proof}

This already implies that there are only finitely many possibilities of what the tropical variety can be under a generic coordinate change which will be needed to prove that in fact, there is only one such possibility in the main theorem. This follows from Theorem \ref{gengrobcomp} together with the fact that the tropical variety always is a subcomplex of the Gr\"obner complex.

\begin{cor}\label{finiteoptions}
Let $I\subset S_{\K}$ be a graded ideal and $V\subset \GL_n({\kx})$ a subvariety. Then there exists a Zariski-open set $\emptyset\neq U\subset V$, such that the tropical variety $T(g(I))$ is one of a finite set of polyhedral complexes for all $g\in U$.
\end{cor}

\section{Generating Systems of Projection Ideals}\label{techgensets}

To prove the existence of generic tropical varieties on a given subvariety $V$ of $\GL_n({\kx})$ some ideas and results from \cite{HETH} need to be generalized. In particular, we want to apply \cite[Theorem 3.5]{HETH} in a generic setting. There an ideal $J$ is defined corresponding to a prime ideal $I$ and a linear projection $\pi$. The idea behind this definition is, that the tropical variety of $J\cap {\K}[x_1,\ldots,x_n]$ essentially is the image of the tropical variety of $I$ under $\pi$, see Proposition \ref{hypecase} (originally \cite[Corollary 3.6]{HETH}). Thus the ideals $J$ and $J\cap {\K}[x_1,\ldots,x_n]$ provide a tool describe the tropical variety of $I$ by the simpler tropical varieties defined by $J$ or $J\cap {\K}[x_1,\ldots,x_n]$. We will loosely refer to these ideals as \emph{projection ideals}.

In our setting all projection ideals $J(g)$ obtained by this construction corresponding to the prime ideals $g(I)$ for $g\in V$ need to be dealt with simultaneously. To handle these it is convenient to consider the extension of $I$ in the polynomial ring over the field extension ${\K}(V)$ of ${\K}$ as given in Definition \ref{gendef}. Then one can do the construction in this polynomial ring as well, defining an ideal $J(y)$ depending on $V$. The results needed are then obtained by evaluating the residue classes of variables $y_{ij}$ at the $g_{ij}$ for a given $g\in V$. For this the connection between $J(y)$ and $J(g)$ for $g\in V$ needs to be established. The aim of this section is to introduce these auxiliary ideals and show: There exists a finite generating system of $J(y)$ and a non-empty open set $U\subset V$, such that if the $g_{ij}$ are substituted for the $y_{ij}$ in these generators, a generating set of $J(g)$ is obtained for every $g\in U$.

By means of notation for a ring $A$ and $l\in \N$ let $$A[x,\lambda,\theta]=A[x_1,\ldots,x_n,\lambda_1,\ldots,\lambda_l,\theta_1,\ldots,\theta_l]$$ be the polynomial ring in $n+2l$ variables over $A$. We fix the following data for the rest of this section:
\begin{itemize}
\item A set of $l$ linearly independent vectors $\left\{u^{(1)},\ldots,u^{(l)}\right\}$ in $\Z^n$,
\item the composite variables $\tau_1,\ldots,\tau_n$ with $$\tau_i=x_i \prod_{u_i^{(j)}\geq 0} \lambda_j^{u_i^{(j)}}\prod_{u_i^{(j)}<0}\theta_j^{-u_i^{(j)}},$$
\item a graded ideal $I\subset S_{\K}={\K}[x_1,\ldots,x_n]$,
\item a subvariety $V\subset \GL_n({\kx})$.
\end{itemize}

All constructions in this section will depend on this data. We first review the definition of the ideal $J$ in \cite[Theorem 3.5]{HETH} and adapt it to our purposes. Since we want to use Gr\"obner basis theory, however, we do not want work in the ring $${\K}[x_1,\ldots,x_n,\lambda_1^{\pm 1},\ldots,\lambda_l^{\pm 1}]$$ from the start as is done in \cite{HETH}, but in the ``large'' polynomial ring ${\K}[x,\lambda,\theta]$.

\begin{notation}\label{jdefn}
Recall the ${\K}$-algebra homomorphism $y$ from Definition \ref{gendef}. We define the following notation for projection ideals
\begin{eqnarray*}
J(y) & = & ({y}(f)(\tau_1,\ldots,\tau_n):f\in I)\subset {\K}(V)[x,\lambda,\theta],\\
\check{J}(y) & = & ({y}(f)(\tau_1,\ldots,\tau_n):f\in I)\subset {\K}[V][x,\lambda,\theta],\\
 J(g)  &=&  (g(f)(\tau_1,\ldots,\tau_n):f\in I)\subset {\K}[x,\lambda,\theta] \text{ for $g\in \GL_n({\kx})$}.
\end{eqnarray*}
\end{notation}

Note that the ideals $J(y)$ and $J(g)$ correspond to the ideals $J$ defined in \cite[Section 3]{HETH}. The ideal $\check{J}(y)$ is of auxiliary purpose for this section (to prove the second claim of Lemma \ref{genofj}) and is of no further importance for us.

Since ${\K}[V][x,\lambda,\theta]$ is noetherian, there exists a finite system of generators among the given generators of $\check{J}(y)$. We fix such a generating system $$\GG=\left\{{y}(f_1)(\tau),\ldots,{y}(f_r)(\tau)\right\}$$ of $\check{J}(y)$ for some ${y}(f_i)(\tau)\in {\K}[V][x,\lambda,\theta]$. Note that this is also a system of generators of $J(y)\subset {\K}(V)[x,\lambda,\theta]$. Since the chosen generators are elements of ${\K}[V][x,\lambda,\theta]$, there are no denominators in the $y_{ij}$ and we can substitute every $g\in V$ for $y$. Hence, this system of generators also defines a set of generators $\GG(g)=\left\{g(f_1)(\tau),\dots,g(f_r)(\tau)\right\}$ of each ideal $J(g)$ for $g\in V$ by the following simple observation.

\begin{lemma}\label{genofj}
For every $g\in V$ we have $J(g)=(g(f_1)(\tau),\dots,g(f_r)(\tau))$. Moreover, there is an open set $\emptyset\neq U_{\GG}\subset V$, such that all polynomials in $$\left\{g(f_1)(\tau),\dots,g(f_r)(\tau)\right\}$$ have the same support for $g\in U_{\GG}$.
\end{lemma}

\begin{proof}
Let $g\in V$ and $g(f)(\tau)$ be one of the generators of $J(g)$ from Notation
\ref{jdefn}. Then ${y}(f)\in \check{J}(y)$, so there exist
$h_1,\ldots,h_r\in {\K}[V][x,\lambda,\theta]$ with $${y}(f)(\tau)=\sum_{i=1}^r
h_i({y}) {y}(f_i)(\tau).$$ This implies
$$g(f)(\tau)=\sum_{i=1}^r h_i(g) g(f_i)(\tau)\in(g(f_1)(\tau),\dots,g(f_r)(\tau)),$$ proving that $J(g)=(g(f_1)(\tau),\dots,g(f_r)(\tau))$. The set $U_{\GG}$ can be chosen as the set of all $g\in V$, such that no coefficient in the $g(f_1)(\tau),\dots,g(f_r)(\tau)$ vanishes.
\end{proof}

Following the procedure in \cite{HETH} we later want to consider the ideals $J(g)\subset {\K}[x,\lambda,\theta]$ in the quotient ring ${\K}[x,\lambda,\lambda^{-1}]$. We need to ensure that in passing from ${\K}(V)[x,\lambda,\theta]$ to the quotient ${\K}(V)[x,\lambda,\lambda^{-1}]$ we keep a finite generating system of the residue ideal of $J(y)$, such that if we substitute ``generic'' $g\in V$ for $y$ we obtain a generating system of the residue ideal of $J(g)$ in ${\K}[x,\lambda,\lambda^{-1}]$. To do this let $W_A\subset A[x,\lambda,\theta]$ be the ideal $W_A=(\lambda_i\theta_i-1: i=1,\ldots,l)$ for $A={\K}$ or $A={\K}(V)$.

We deal with the above problem for the ideals $J(y)+W_{{\K}(V)}$ and $J(g)+W_{\K}$ using Gr\"obner basis theory. The idea is to guarantee that the Buchberger algorithm applied to generators of $J(y)+W_{{\K}(V)}$ consists of exactly the same computational steps as if it is applied to the corresponding generators $J(g)+W_{\K}$ generically.

Let $\succ$ be the lexicographic term order on $A[x,\lambda,\theta]$ induced by $$\lambda_1\succ\ldots\succ\lambda_l\succ\theta_1\succ\ldots\succ\theta_l\succ x_1\succ\ldots\succ x_n.$$ Recall that this is an elimination order with respect to the variables $\lambda_1,\ldots,\lambda_l,\theta_1,\ldots,\theta_l$, see \cite[p. 361, Example 2]{E}.

\begin{lemma}\label{groebbigideals}
There exists an open subset $\emptyset \neq U_{\Gc}\subset
V$ and polynomials $$h_1({y}),\ldots,h_s({y})\in J(y)+W_{{\K}(V)}\subset
{\K}(V)[x,\lambda,\theta],$$ such that:
\begin{enumerate}
\item
$\Gc=\left\{h_1({y}),\ldots,h_s({y})\right\}$ is the reduced Gr\"obner
basis of the ideal $J(y)+W_{{\K}(V)}$ in ${\K}(V)[x,\lambda,\theta]$ with respect
to $\succ$.
\item
$\Gc(g)=\left\{h_1(g),\ldots,h_s(g)\right\}$ is the reduced
Gr\"obner basis of the ideal $J(g)+W_{{\K}}$ in ${\K}[x,\lambda,\theta]$
with respect to $\succ$ for all $g\in U_{\Gc}$.
\item
The set $\Gc$ and all the sets $\Gc(g)$ for $g\in U_{\Gc}$ have the same support.
\end{enumerate}
\end{lemma}

\begin{proof}
We start with the finite generating sets $\GG\cup\left\{\lambda_1\theta_1-1,\ldots,\lambda_l\theta_l-1\right\}$ of the ideal
$J(y)+W_{{\K}(V)}$ and $\GG(g)\cup\left\{\lambda_1\theta_1-1,\ldots,\lambda_l\theta_l-1\right\}$
of $J(g)+W_{{\K}}$, all of which have the same support for all $g\in U_{\GG}$ as in Lemma \ref{genofj}. Proceeding by applying the Buchberger Algorithm we compute the reduced Gr\"obner basis $\left\{h_1({y}),\ldots,h_s({y})\right\}$ of $J(y)+W_{{\K}(V)}$ with respect to $\succ$. In each of the finitely many steps finitely many polynomials appear, which all have quotients of residue classes of polynomials in the
${y_{ij}}$ as coefficients. Choose $U_{\Gc}\subset U_{\GG}\subset V$, such that none of these residue classes
vanishes for any $g\in U_{\Gc}$. Then $U_{\Gc}$ is non-empty and open and we have $\Gc(g)=\left\{h_1(g),\ldots,h_s(g)\right\}$ is a reduced
Gr\"obner basis of $J(g)+W_{{\K}}$ with the same support.
\end{proof}

The ideals defining the tropical hypersurfaces used to express tropical varieties in Section \ref{generictropicalvarieties} are the intersections of the quotient ideals of $J(g)+W_{\K}$ in ${\K}[x,\lambda,\theta]/W_{\K}$ with $S_{\K}={\K}[x_1,\ldots,x_n]$:

\begin{notation}\label{tilddefn}
For $A={\K}$ oder $A={\K}(V)$ let $$\varphi_A: A[x,\lambda,\theta]\longrightarrow
A[x,\lambda,\theta]/W_A$$ be the canonical ring epimorphism. Consider
the images $\varphi_{{\K}(V)}(J(y))$ and $\varphi_{{\K}}(J(g))$. Then we denote the ideal
$\varphi_{{\K}(V)}(J(y))\cap S_{{\K}(V)}$ by $\tilde{J}(y)$ and the ideal $\varphi_{{\K}}(J(g))\cap S_{\K}$ by $\tilde{J}(g)$.
\end{notation}

Note that $\tilde{J}(g)\subset S_{\K}$ is exactly the ideal $J\cap {\K}[x_1,\ldots,x_n]$ as defined in \cite[Section 3]{HETH} corresponding to the ideal $g(I)$ instead of $I$ for $g\in V$. In particular, we have the following result, which has been proved in \cite[Lemma 3.3]{HETH}.

\begin{lemma}\label{JsubI}
With the notation from above we have that $\tilde{J}(y)\subset y(I)$ and $\tilde{J}(g)\subset g(I)$ for every $g\in\GL_n({\kx})$.
\end{lemma}


By the definition of $W_A$ there is a canonical $A$-algebra isomorphism between $A[x,\lambda,\theta]/W_A$ and $A[x,\lambda,\lambda^{-1}]$ for $A={\K}$ or $A={\K}(V)$. The elements of $A[x,\lambda,\theta]/W_A$ can thus be thought of as polynomial expressions in the $x$,$\lambda$ and $\lambda^{-1}$ rather than as residue classes. Moreover, $A[x,\lambda,\lambda^{-1}]$ is a localization of $A[x,\lambda]$, which will be of use in the following statement. Note that the polynomial ring $S_A=A[x_1,\ldots,x_n]\subset
A[x,\lambda,\theta]$ is mapped injectively to $\varphi_A(S_A)\subset A[x,\lambda,\lambda^{-1}]$, since $W_A\cap S_A=\left\{0\right\}$.
Therefore we can identify $S_A$ with $\varphi_A(S_A)$. 

For the proof of our main theorem we need that for $g\in\GL_n({\kx})$ the ideals $\tilde{J}(y)\subset S_{\K(V)}$ and $\tilde{J}(g)\subset S_{\K}$ are prime if $I\subset S_{\K}$ is prime. A version of this has also been proved in \cite[Theorem 3.11]{HEP}. We include the proof of this statement in our setting.

\begin{lemma}\label{jtildeprime}
Let $I\subset S_{\K}={\K}[x_1,\ldots,x_n]$ be a graded prime ideal. Then the ideal $\tilde{J}(y)$ in $S_{\K(V)}$ is also prime. Moreover, all ideals $\tilde{J}(g)$ are prime for $g\in\GL_n({\kx})$.
\end{lemma}

\begin{proof}
We will prove the statement for the ideal $\tilde{J}(y)$ in $S_{\K(V)}$. The proof for the ideals $\tilde{J}(g)$ for $g\in\GL_n({\kx})$ can be done analogously. In the first part of the proof we will consider ${\K(V)}[x,\lambda,\lambda^{-1}]$ as an ${\K(V)}[x]$-algebra and denote it by $S_{\K(V)}[\lambda,\lambda^{-1}]$ to display this. In the second part ${\K(V)}[x,\lambda,\lambda^{-1}]$ will be considered as an ${\K(V)}[\lambda,\lambda^{-1}]$-algebra. To emphasize this we will denote it by ${\K(V)}[\lambda,\lambda^{-1}][x]$, when we do so.

Consider the chain of inclusions $$S_{\K}\hookrightarrow S_{\K(V)}\hookrightarrow S_{\K(V)}[\lambda]\hookrightarrow S_{\K(V)}[\lambda,\lambda^{-1}],$$ where the first one is given by $y$ as in Definition \ref{gendef} and the other ones are the natural ones. As $I\subset S_{\K}$ is prime, by Proposition \ref{extkrempel} the extension $y(I)\subset S_{\K(V)}$ is prime (this step is superfluous in the corresponding proof for the $\tilde{J}(g)$). Since $S_{\K(V)}[\lambda]$ is a polynomial ring over $S_{\K(V)}$, the residue ring $S_{\K(V)}[\lambda]/y(I)S_{\K(V)}[\lambda]$ is isomorphic to $(S_{\K(V)}/y(I))[\lambda]$. As $(S_{\K(V)}/y(I))[\lambda]$ is an integral domain, so is $S_{\K(V)}[\lambda]/y(I)S_{\K(V)}[\lambda]$. Hence, $y(I)S_{\K(V)}[\lambda]$ is prime. By \cite[Proposition 2.2(b)]{E} this implies that $y(I)S_{\K(V)}[\lambda,\lambda^{-1}]$ is prime.

Let $\psi: {\K(V)}[\lambda,\lambda^{-1}][x] \rightarrow {\K(V)}[\lambda,\lambda^{-1}][x]$ be the map of ${\K(V)}[\lambda,\lambda^{-1}]$-algebras induced by $$\psi(x_i)=x_i \prod_{j=1}^{l}\lambda_j^{u_i^{(j)}}.$$ This map is an isomorphism with the inverse given by mapping $x_i$ to $\prod_{j=1}^{l}\lambda_j^{-u_i^{(j)}}$. 

With the canonical identification of ${\K(V)}[x,\lambda,\lambda^{-1}]$ with ${\K(V)}[x,\lambda,\theta]/W_{\K(V)}$ the ideal $$\psi (y(I)S_{\K(V)}[\lambda,\lambda^{-1}])\subset {\K(V)}[x,\lambda,\lambda^{-1}]$$ is exactly the ideal $$\varphi_{\K(V)}(J(y))\subset {\K(V)}[x,\lambda,\theta]/W_{\K(V)}$$ as in Notation \ref{tilddefn}. Since $y(I)S_{\K(V)}[\lambda,\lambda^{-1}]$ is prime by the above argument and $\psi$ is an isomorphism, the ideal $\psi(y(I)S_{\K(V)}[\lambda,\lambda^{-1}])$ is also prime. Thus $\varphi_{\K(V)}(J(y))$ is prime. This also means that the intersection $\varphi_{\K(V)}(J(y))\cap S_{\K(V)}$ is prime, which by definition is the ideal $\tilde{J}(y)\subset S_{\K(V)}$.
\end{proof}

The final aim of this section is to compute Gr\"obner bases of the same support of the ideals $\tilde{J}(y)$ and $\tilde{J}(g)$ using elimination with respect to the variables $\lambda_1,\ldots,\lambda_l,\theta_1,\ldots,\theta_l$. To be able to apply this idea to the ideals $\varphi_{{\K}(V)}(J(y))$ and $\varphi_{{\K}}(J(g))$ we need to show that these have the same intersections with $S_{{\K}(V)}$ and $S_{\K}$ respectively as the ideals in Lemma \ref{groebbigideals}:

\begin{lemma}\label{sameideal}
With the above notation $(W_{{\K}(V)}+J(y))\cap S_{{\K}(V)}=\varphi_{{\K}(V)}(J(y))\cap S_{{\K}(V)}=\tilde{J}(y)$ and $(W_{{\K}}+J(g))\cap S_{\K}= \varphi_{{\K}}(J(g))\cap S_{\K}=\tilde{J}(g)$.
\end{lemma}

\begin{proof}
Since the proof does not depend on the chosen field ${\K}(V)$ or ${\K}$, it
suffices to show the first statement. The second one is proved in
exactly the same way. For simplicity we denote $\varphi_{{\K}(V)}$ by $\varphi$ and write $J$ for $J(y)$ as well as $W$ for $W_{{\K}(V)}$ and $S$ for $S_{{\K}(V)}$ in this proof.

Let $h\in(W+J)\cap S$. Then $\varphi(h)=h$, since $h\in S$ is
independent of $\lambda$ and $\theta$. On the other hand we can write
$h=h_W+h_J$, where $h_W\in W$ and $h_J\in J$. Then we have
$$\varphi(h)=\varphi(h_W)+\varphi(h_J)=0+\varphi(h_J)\in
\varphi(J).$$ Hence, $h\in \varphi(J)\cap S$.

For the other inclusion let $h\in \varphi(J)\cap S$. Since
$h\in \varphi(J)$, there exists $b\in J$ with $\varphi(b)=h$.
The aim is to construct an element $\tilde{b}\in (W+J)\cap S$ by adding a suitable element of $W$ to $b$. Then we know
$\varphi(\tilde{b})=\varphi(b)=h$, and both $\tilde{b}$ and $h$ are
independent of $\lambda$ and $\theta$, which implies $\tilde{b}=h$. To
find $\tilde{b}$ write $b=\sum_{(\nu_1,\nu_2)\in \N^{2l}}
c_{\nu}(x,y) \lambda^{\nu_1}\theta^{\nu_2}$ as a polynomial in the
$\lambda$ and $\theta$. Then we have $$h=\varphi(b)=\sum_{a\in \Z^{2l}}
\left(\sum_{\nu_1-\nu_2=a} c_{\nu}(x,y)\right) \lambda^a\in S.$$ If $a\neq 0$,
then $(\sum_{\nu_1-\nu_2=a} c_{\nu}(x,y))=0$. Writing
$$b=\sum_{\nu_1-\nu_2\neq a}
c_{\nu}(x,y)\lambda^{\nu_1}\theta^{\nu_2}+\sum_{\nu_1-\nu_2=a}
c_{\nu}(x,y)\lambda^{\nu_1}\theta^{\nu_2}$$ the second part must be
contained in $W$, since it maps to $0$ under $\varphi$. So we can drop
the second part from $b$ and this still maps to $h$ under $\varphi$
and is an element from $W+J$. Without loss of generality it can thus
be assumed that $b$ is a polynomial in $x_1,\ldots,x_n$ and
$\lambda_1\theta_1,\ldots,\lambda_l\theta_l$, since it only contains terms
$c_{\nu}(x,y)\lambda^{\nu_1}\theta^{\nu_2}$ with $\nu_1=\nu_2$.

To eliminate the $\lambda_j\theta_j$ from $b$ observe that $(\lambda_j\theta_j)^{d}-1\in W$ for every $d>0$. Indeed $$(\lambda_j\theta_j)^{d}-1=\left(\sum_{s=0}^{{d}-1}
(\lambda_j\theta_j)^s\right)(\lambda_j\theta_j-1).$$ For every term
$c(x,y)(\lambda_j\theta_j)^{t_j}\prod_{i\neq j}(\lambda_i\theta_i)^{t_i}$ we can subtract
$$c(x,y)
\left((\lambda_j\theta_j)^{t_j}-1\right) \prod_{i\neq j}(\lambda_i\theta_i)^{t_i}\in
W$$ from $b$ and thus eliminate the variable $\lambda_j\theta_j$ from
this term. Doing this inductively for all $\lambda_j\theta_j$,
$j=1,\ldots,l$ in all terms of $b$ we obtain the expression $\tilde{b}=b+b_W\in S$,
for an element $b_W\in W$. Hence, $\tilde{b}\in (J+W)\cap S$ proving the claim.
\end{proof}

By elimination we can now find a Gr\"obner basis for $\tilde{J}(y)\subset S_{{\K}(V)}$, such that if we substitute $g\in U_{\Gc}$ from Lemma \ref{groebbigideals}  we obtain a Gr\"obner basis of $\tilde{J}(g)\subset S$. Let $\Gc$ be the reduced Gr\"obner basis of $J(y)+W_{{\K}(V)}\subset
{\K}(V)[x,\lambda,\theta]$ with respect to the lexicographic term order $\succ$ and let $\Gc(g)$ the reduced Gr\"obner bases of
$J(g)+W_{{\K}}\subset {\K}[x,\lambda,\theta]$ with respect to $\succ$ for all $g\in U_{\Gc}$; all as in Lemma \ref{groebbigideals}. 

\begin{cor}\label{groebsmallideal}
With the above notation we have:
\begin{enumerate}
\item $\tilde{\Gc}=\Gc\cap S_{{\K}(V)}$ is a Gr\"obner basis of $\tilde{J}(y)\subset S_{{\K}(V)}$ with respect to
$\succ$.
\item $\tilde{\Gc}(g)=\Gc(g)\cap S_{\K}$ is a Gr\"obner basis of $\tilde{J}(g)\subset S_{\K}$ with respect to $\succ$ for all $g\in
U_{\Gc}$.
\item The set $\tilde{\Gc}$ and all the sets $\tilde{\Gc}(g)$ for $g\in U_{\Gc}$ have the same
support.
\end{enumerate}
\end{cor}

\begin{proof}
By elimination \cite[Proposition 15.29]{E} $\tilde{\Gc}$ and $\tilde{\Gc}(g)$
are Gr\"obner bases of $(W_{{\K}(V)}+J(y))\cap S_{{\K}(V)}$ and $(W_{{\K}}+J(g))\cap S_{\K}$, respectively. By Lemma \ref{sameideal} we know that $(W_{{\K}(V)}+J(y))\cap
S_{{\K}(V)}=\tilde{J}(y)$ and $(W_{{\K}}+J(g))\cap S_{\K}=\tilde{J}(g)$. Finally, (iii)
follows from Lemma \ref{groebbigideals} (iii).
\end{proof}

In particular, we have found a generating system $\tilde{\Gc}$ of
$\tilde{J}(y)$ and a non-empty open subset of
$V$, such that if we substitute the residue classes of the ${y_{ij}}$ modulo the defining ideal of $V$
by $g_{ij}$ in this set, we obtain a generating system of the ideal
$\tilde{J}(g)$. Moreover, we get the following simple corollary on the Krull dimensions of projection ideals, which will be useful later.

\begin{cor}\label{liftdim}
For all $g\in U_{\Gc}$ as defined in Lemma \ref{groebbigideals} we have $$\dim (S_{\K}/\tilde{J}(g))= \dim (S_{{\K}(V)}/\tilde{J}(y))$$ for the ideals $\tilde{J}(g)\subset S_{\K}$ and $\tilde{J}(y)\subset S_{{\K}(V)}$.
\end{cor}

\begin{proof}
We have Gr\"obner bases $\tilde{\Gc}(g)$ of $\tilde{J}(g)$ and
$\tilde{\Gc}$ of $\tilde{J}(y)$ with respect to some term order $\succ$ with the same support as stated in
Corollary \ref{groebsmallideal}. This implies that
$\inom_{\succ}(\tilde{J}(g))\subset S_{\K}$ and
$\inom_{\succ}(\tilde{J}(y))\subset S_{{\K}(V)}$ are generated by the same
monomials. Since the dimension of monomial ideals does not depend on
the ground field, we thus have $$\dim (S_{\K}/\tilde{J}(g))=\dim (S_{\K}/\inom_{\succ} (\tilde{J}(g)))=\dim (S_{{\K}(V)}/\inom_{\succ} (\tilde{J}(y)))= \dim (S_{{\K}(V)}/\tilde{J}(y)).$$
\end{proof}

\section{Rational Projections}\label{ratproj}

The main tool to express a tropical variety as an intersection of tropical hypersurfaces as done in \cite{HETH} are certain linear projections
. The idea is to first project a tropical variety, such that the dimension of the ambient space is as small as possible, but no information on the structure of the tropical variety is lost. For an $m$-dimensional tropical variety in $\R^n$ it turns out, that a linear map from $\R^n$ to $\R^{m+1}$ can be used for this. As the kernel of such a map is $(n-m-1)$-dimensional, the inverse image of an $m$-dimensional tropical variety is a finite set of polyhedra of dimension $n-1$ in $\R^n$. In \cite{HETH} it is shown that this inverse image is a tropical hypersurface. We then need to recover the original tropical variety from tropical hypersurfaces obtained by projections as described above. This is solved in \cite{HETH} by applying a theorem of Bieri and Groves \cite[Theorem 4.4]{BIGR}. 

To proceed in the same way as done in \cite{HETH} we need a version of \cite[Theorem 4.4]{BIGR} for finitely many subsets of $\R^n$ instead of only one.

\begin{defn}
Let $m<n$ be positive integers and
\begin{eqnarray*}
\pi: \R^n & \longrightarrow & \R^{m+1}\\
x         & \longmapsto     & Ax
\end{eqnarray*}
be a linear map with rational matrix $A$ of maximal possible rank.
Such a map will simply be called a \emph{rational projection}. Let
$\Pi$ be the set of equivalence classes of all rational projections
with respect to the equivalence relation $$\pi \sim \pi'
\Longleftrightarrow \ker \pi=\ker \pi'.$$ A vector subspace of
$\R^n$ will be called \emph{rational}, if it has a basis of rational
vectors.
\end{defn}

Note that $\ker \pi$ is rational for $\pi\in \Pi$. Moreover, $\Pi$ can be identified with $$\left\{U\subset \R^n: U
\text{ is a rational vector subspace of $\R^n$, }  \dim U=n-m-1\right\}.$$ As in \cite[Section 4.1]{BIGR} the topology on $\Pi$ will be the one induced by the Zariski topology of the Grassmannian $\Gr_{\R}(n-m-1,n)$. Thus $\Pi$ is a dense subset of $\Gr_{\R}(n-m-1,n)$ consisting of all rational vector subspaces of $\Gr_{\R}(n-m-1,n)$. 

Since every open subset of $\Pi$ is by definition the intersection of an open subset of $\Pi$ with $\Gr_{\R}(n-m-1,n)$ and $\Pi$ is dense in $\Gr_{\R}(n-m-1,n)$, it follows that every non-empty open set in $\Pi$ is dense. Note that \cite[Lemma 4.2, Lemma 4.3 and Theorem 4.4]{BIGR}, which consider the set of all (not necessarily rational) projections hold for $\Pi$ with the above topology as well. In particular, all of the following statements are well-defined on the equivalence classes of $\Pi$, although they concern representatives of these.

One necessary condition to be able to recover a tropical variety from its image under a rational projection $\pi$ is that $\pi$ preserves the dimension of all polyhedra in the tropical variety. This is true for almost all rational projections as will be the content of the next statement. It is a direct application of \cite[Lemma 4.2]{BIGR}, see also \cite[Section 5]{PA} for a related result.

\begin{lemma}[\cite{BIGR}]\label{dense}
Let $m<n$ and $\Dc=\left\{P_1,\ldots,P_t\right\}$ be a finite collection of $m$-dimensional polyhedra in $\R^n$. Then the set of all rational projections $\pi: \R^n\longrightarrow \R^{m+1}$, such that $\dim\pi(P_i)=m$ for $i=1,\ldots,t$, contains an open and dense subset $\tilde{D}\subset \Pi$ in the set of all rational projections.
\end{lemma}

The key to recover tropical varieties from their images under rational projections is \cite[Theorem 4.4]{BIGR}. In \cite{HETH} this theorem is directly applied to recover a single tropical variety. To handle the generic case it is necessary to be able to apply \cite[Theorem 4.4]{BIGR} to all possible tropical varieties under a generic coordinate change. However, there is only a finite number of possibilities of what the tropical variety of an ideal can be generically, see Corollary \ref{finiteoptions}. This hints at the necessity of the following version of \cite[Theorem 4.4]{BIGR} for our purposes, which can be proven in the same fashion as the original theorem.

\begin{theorem}[\cite{BIGR}]\label{bierigrovesorig}
Let $A_1,\ldots,A_t\subset \R^n$ be arbitrary subsets and assume that there exists a dense set $D'\subset \Pi$ of rational
projections $\pi: \R^n\longrightarrow \R^{m+1}$, such that $\pi(A_j)$ is a finite union of polyhedra of dimension less than or equal to $m$ for every $j\in \left\{1,\ldots,t\right\}$. Then there exist $\pi_0,\ldots,\pi_n\in D'$, such that for every $j$ we have $$A_j=\bigcap_{i=0}^n \pi_i^{-1} \pi_i(A_j).$$
\end{theorem}

\section{Generic Tropical Varieties}\label{generictropicalvarieties}

In this section the existence of the generic tropical variety for a graded ideal $I$ on a subvariety $V$ of $\GL_n({\kx})$ will be established. We first prove this for graded prime ideals using the methods in \cite[Section 3]{HETH}, and then generalize this to arbitrary graded ideals. The idea will be to construct finitely many polynomials $F_i(g)\in g(I)$ with constant tropical variety on a Zariski-open subset of $\emptyset\neq V\subset \GL_n({\kx})$ for which $T(g(I))=\bigcap_{i}T(F_i(g))$ for $g\in V$.

This amounts to giving a generic version of \cite[Corollary 3.6]{HETH} explained below. To do this we need the ideals which are associated to $I$ in \cite{HETH} to describe the tropical variety of $I$ as an intersection of tropical hypersurfaces. We must deal with these ideals corresponding to $g(I)$ for all $g$ in some non-empty open subset of $V\subset\GL_n({\kx})$ simultaneously. The main technical treatment for this was done in Section \ref{techgensets}, where we obtained generating systems with the necessary properties for these ideals. 

Let ${\kx}\hookrightarrow {\K}$ be as in Notation \ref{kandK} and $I\subset S_L$ be a graded ideal. Let $\pi: \R^n  \to \R^{m+1}$ be a rational projection in $\Pi$ as in Section \ref{ratproj}. Fix a basis $u^{(1)},\ldots,u^{(l)}\in \Z^n$ for $l=n-(m+1)$ of $\ker \pi$. For $g\in \GL_n({\kx})$ we consider the ideal $$\tilde{J}(g)=\varphi_{\K}(J(g))\cap {\K}[x_1,\ldots,x_n]\subset {\K}[x_1,\ldots,x_n]$$ as in Notation \ref{tilddefn} a priori depending on the chosen basis $u^{(1)},\ldots,u^{(l)}$. If $g$ is the identity in $\GL_n({\kx})$, this is exactly the ideal $J\cap {\K}[x_1,\ldots,x_n]$ from \cite[Theorem 3.1]{HETH}, which provided the idea for the definition of its generic versions in Section \ref{techgensets}. 

We first cite an important result from \cite{HETH}, which will establish the connection between the ideals of Section \ref{techgensets} and the tropical variety of $I$. This allows us to express $T(I)$ as an intersection of tropical hypersurfaces.

\begin{prop}[{\cite[Corollary 3.6]{HETH}}]\label{hypecase}
Let $I\subset S_{\K}={\K}[x_1,\ldots,x_n]$ be a graded prime ideal with $\dim (S_{\K}/I)=m$. Then there exists a dense open subset $D\subset \Pi$, such that for $\pi\in D$ we have: If $\dim \pi(P)=m$ for every maximal polyhedron $P$ in $T(g(I))$, then $$T(\tilde{J}(g))=\pi^{-1} \pi(T(g(I)))$$ is a tropical hypersurface.
\end{prop}

First of all we assert that for prime ideals the condition in Proposition \ref{hypecase} on the dimension of the image of the maximal polyhedra of $T(I)$ under $\pi$ can be met generically.

\begin{rem}\label{dimgen}
Let $I\subset S_{\K}={\K}[x_1,\ldots,x_n]$ be a graded prime ideal with $\dim (S_{\K}/I)=m>0$ and $V\subset \GL_n({\kx})$ a subvariety. From Corollary \ref{finiteoptions} we know that there exists an open subset $\emptyset\neq U\subset V$, such that $T(g(I))$ is one of finitely many $m$-dimensional polyhedral complexes $\Fc_1,\ldots,\Fc_s$ for all $g\in U$. All these complexes are pure, as $I$ is prime. By Lemma \ref{dense} there exists an open and dense set $\tilde{D}\subset \Pi$, such that $\dim\pi(P)=m$ for every $m$-dimensional polyhedron $P$ in any of the $\Fc_k$ for every $\pi \in \tilde{D}$. As both $\tilde{D}$ and $D$ (from Proposition \ref{hypecase}) are open and dense in $\Pi$, so is $D'=\tilde{D}\cap D$. So for every $\pi\in D'$ and every $g\in U$ we have $\dim \pi(P)=m$ for every maximal polyhedron $P$ in $T(g(I))$.
\end{rem}

To handle the ideals $\tilde{J}(g)$ for all $g\in V$ simultaneously we have also constructed the ideal $\tilde{J}(y)\subset {\K}(V)[x_1,\ldots,x_n]$ in Notation \ref{tilddefn}. We will mainly need one important fact about all these ideals: There is a finite generating system of $\tilde{J}(y)$ and a non-empty open subset $U\subset V$, such that substituting the $g_{ij}$ for the variables $y_{ij}$ in the generators yields a finite generating set of $\tilde{J}(g)$ for every $g\in U$, all proved in Corollary \ref{groebsmallideal}.

In Theorem \ref{tropbasiselement}, which is the technical key statement in this section, we will show that the ideal $\tilde{J}(y)$ is principal and we want to substitute $g\in V$ into the given generator in ${\K}(V)[x_1,\ldots,x_n]$. For each $g$ which can be substituted this yields a polynomial in ${\K}[x_1,\ldots,x_n]$. The tropical hypersurfaces defined by these polynomials are generically all the same, as will be shown in the following lemma.

\begin{lemma}\label{trophypeconst}
Let $V\subset \GL_n({\kx})$ be a subvariety and $F(y)\in {\K}(V)[x_1,\ldots,x_n]$ be a homogeneous polynomial. Then there exists an open subset $\emptyset\neq U\subset V$, such that $T(F(g))$ is the same (possibly empty) polyhedral complex for all $g\in U$.
\end{lemma}

\begin{proof}
Let $F(y)=\sum_{\eta} (\frac{f_{\eta}}{h_{\eta}})x^{\eta}$, where $f_{\eta}$ and $h_{\eta}$ are elements of ${\K}[V]={\K}[Y]/P{\K}[Y]$ as used in Notation \ref{kv}. Recall that $f_{\eta}$ and $h_{\eta}$ define functions from $V$ to ${\K}$. Consider $F(y)$ as a polynomial in ${\kx}(V)\left\{\!\!\left\{t\right\}\!\!\right\}[x_1,\ldots,x_n]$ via the canonical inclusion, see Corollary \ref{valfieldext} and Notation \ref{kv}. Thus we can write $$F(y)=\sum_{\eta}\left(\sum_{\mu\in\R} \left(\frac{\sum_{\nu} a^{\eta}_{\nu\mu}y^{\nu}+P}{\sum_{\nu'}b^{\eta}_{\nu'\mu}y^{\nu'}+P}\right)t^{\mu}\right)x^{\eta},$$ where all $a^{\eta}_{\nu\mu}$ and $b^{\eta}_{\nu'\mu}$ are elements of ${\kx}$. 
For every $\eta$ let $$\mu^{\eta}_{\min}=v\left(\sum_{\mu\in\R} \left(\frac{\sum_{\nu} a^{\eta}_{\nu\mu}y^{\nu}+P}{\sum_{\nu'}b^{\eta}_{\nu'\mu}y^{\nu'}+P}\right)t^{\mu}\right)$$ be the valuation of the coefficient of $x^{\eta}$.

Choose $U\subset V$ to be a non-empty open subset, such that for $g\in U$ we have:
\begin{enumerate}
\item $h_{\eta}(g)\neq 0$ for every $\eta$ appearing. This ensures that we can substitute $g$ into $F(y)$.
\item $f_{\eta}(g)\neq 0$ for every $\eta$. Thus $F(g)$ is a polynomial in ${\K}[x_1,\ldots,x_n]$ with the same support for all $g\in U$.
\item $\sum_{\nu} a^{\eta}_{\nu\mu^{\eta}_{\min}}g^{\nu}\neq 0$ and $\sum_{\nu'}b^{\eta}_{\nu'\mu^{\eta}_{\min}}g^{\nu'}\neq0$. This guarantees that for a given $\eta$ and $\omega\in \R^n$ the expression $$v\left(\frac{f_{\eta}(g)}{h_{\eta}(g)}\right)+\eta\cdot \omega\in \R$$ is the same for every $g\in U$.
\end{enumerate}
As the tropical hypersurface of $F(g)$ depends exactly on this data, we have that $T(F(g))$ is the same polyhedral complex in $\R^n$ for all $g\in U$. This complex is empty, if and only if $F(y)$ is a monomial.
\end{proof}

With this we can now prove a general version of Proposition \ref{hypecase}, which will be the crucial step in the proof of the existence of generic tropical varieties.

\begin{theorem}\label{tropbasiselement}
Let $I\subset S_{\K}={\K}[x_1,\ldots,x_n]$ be a graded prime ideal with $\dim (S_{\K}/I)=m$, let $V\subset \GL_n({\kx})$ be a subvariety and $\pi:\R^n\longrightarrow \R^{m+1}$ a rational projection in $D'\subset \Pi$ as defined in Remark \ref{dimgen}. Then
\begin{enumerate}
\item either $T(g(I))=\emptyset$ for all $g\in V$
\item or there exists $F(y)\in y(I)\subset {\K}(V)[x_1,\ldots,x_n]$ and an open subset $\emptyset \neq U\subset V$, such that $T(F(g))$ is the same polyhedral complex for all $g\in U$ and the set $\pi^{-1}\pi(T(g(I)))$ is the (underlying set of) tropical hypersurface defined by $F(g)\in g(I)$.
\end{enumerate}
\end{theorem}

\begin{proof}
If $T(g(I))=\emptyset$ for all $g\in V$, there is nothing to prove. Assume there exists a $\hat{g}\in V$, such that $T(\hat{g}(I))\neq\emptyset$. In particular, $\dim (S_{\K}/I)>0$ in this case. The idea of the proof is to obtain a polynomial $F(y)\in {\K}(V)[x_1,\ldots,x_n]$, such that $\tilde{J}(y)=(F(y))$ is a principal ideal. Then we want to choose $U\subset V$, such that if we substitute the coefficients $g_{ij}$ for $y_{ij}$ for $g\in U$, we get that $T(\tilde{J}(g))=T(F(g))$ is the same tropical hypersurface.

Since $I$ is prime, the tropical variety $T(\hat{g}(I))$ of $\hat{g}(I)$ is a pure $m$-dimensional polyhedral complex. By Lemma \ref{dense} and Proposition \ref{hypecase} there exists a projection $\rho\in \Pi$, such that the tropical variety $T(\tilde{J}(\hat{g}))=\rho^{-1}\rho(T(\hat{g}(I)))$ is a tropical hypersurface. As $\tilde{J}(\hat{g})$ is prime by Lemma \ref{jtildeprime}, we have $\dim (S_{\K}/\tilde{J}(\hat{g}))=n-1$. Thus $\dim(S_{{\K}(V)}/\tilde{J}(y))=n-1$ by Corollary \ref{liftdim}. In addition, again by Lemma \ref{jtildeprime} the ideal $\tilde{J}(y)\subset {\K}(V)[x_1,\ldots,x_n]$ is a prime ideal. So $\tilde{J}(y)$ is a principal ideal, as it is prime and of height $1$. This shows that $\tilde{J}(y)=(F(y))$ for a non-zero homogeneous polynomial $F(y)\in {\K}(V)[x_1,\ldots,x_n]$. Note that by Lemma \ref{JsubI} we have indeed $F(y)\in y(I)$.

The next aim is to substitute appropriate $g$ for the $y$ in $F(y)$, such that the conditions in (ii) are fulfilled. In Corollary \ref{groebsmallideal} we have obtained a finite generating set $\tilde{\Gc}$ of the ideal $\tilde{J}(y)\subset {\K}(V)[x_1,\ldots,x_n]$, such that if we substitute $g$ in some non-empty Zariski-open set $U_{\Gc}\subset V$, then $\tilde{\Gc}(g)$ is a generating set of $\tilde{J}(g)\subset {\K}[x_1,\ldots,x_n]$. We have $$\tilde{J}(y)=(\tilde{\Gc})=(f_1(y),\ldots,f_q(y))=(F(y)),$$ so we can write $f_j(y)=r_{j}(y) F(y)$ and $F(y)=\sum_{j=0}^q s_j(y) f_j(y)$ for some polynomials $r_{j}(y),s_j(y)\in {\K}(V)[x_1,\ldots,x_n]$  for every $j=1,\ldots,q$. Choose $\tilde{U}\subset U_{\Gc}$ to be a non-empty open subset of $V$, such that for $g\in \tilde{U}$ no denominator in any of the coefficients of $F(y)$ and in any of the $r_{j}(y)$ and $s_j(y)$ for $j=1,\ldots,q$ vanishes. This condition implies that $\tilde{J}(g)=(F(g))$, hence, $T(F(g))=T(\tilde{J}(g))$ for all $g\in \tilde{U}$. Moreover, by Lemma \ref{trophypeconst} we can choose a non-empty open subset $U'\subset \tilde{U}\subset V$, such that $T(F(g))$ is the same polyhedral complex for all $g\in U'$. In addition, $F(g)\in g(I)$ by Lemma \ref{JsubI}.

By Proposition \ref{hypecase} together with Remark \ref{dimgen} there exists an open subset $\emptyset\neq U''\subset V$, such that $T(\tilde{J}(g))=\pi^{-1} \pi(T(g(I)))$. Hence, for every $g\in U=U'\cap U''$ all the conditions in (ii) are met, which proves the claim.
\end{proof}

In the previous statement it was shown that for a given rational projection $\pi$ under certain conditions $\pi^{-1}\pi(T(g(I)))$ is the same tropical hypersurface for almost all choices of coordinates $g$. We will now use a theorem by Bieri and Groves (in the version stated as Theorem \ref{bierigrovesorig}) to show that in the generic case $T(g(I))$ is cut out by finitely many rational projections. This proves the main result for the case of graded prime ideals, including the existence of generic tropical bases as defined in Definition \ref{gentropbasis}.

\begin{theorem}\label{thm:mainprime}
Let $I\subset S_{\K}={\K}[x_1,\ldots,x_n]$ be a prime ideal and $\emptyset\neq V\subset \GL_n({\kx})$ a closed subvariety.
Then there exists a non-empty Zariski-open set $U\subset V$, such that $T(g(I))$ is the same (possibly empty) polyhedral complex $\gT_V(I)$ for all $g\in U$. Moreover, if $\gT_V(I)\neq \emptyset$ there exists a generic tropical basis of $I$.
\end{theorem}

\begin{proof}
By Theorem \ref{tropbasics} for each $g\in V$ the tropical variety $T(g(I))$ is either empty or pure of dimension $m$.
From Corollary \ref{finiteoptions} we know that there is a non-empty open subset $\tilde{U}\subset V$, such that if $T(g(I))\neq \emptyset$, it is one of finitely many pure $m$-dimensional polyhedral complexes $\left\{\Fc_1,\ldots,\Fc_t\right\}$ for all $g\in
\tilde{U}$, but this complex is not a priori independent of the chosen $g$.

Since the set $D'\subset \Pi$ as defined in Remark \ref{dimgen} is open and dense in $\Pi$, by Theorem \ref{bierigrovesorig} there exist rational projections $\pi_0,\ldots,\pi_n\in D'$, such that $$\Fc_k=\bigcap_{i=0}^n \pi_i^{-1} \pi_i(\Fc_k)$$ for every $k$. For every $i=0,\ldots,n$ there exist $F_i(y)\in {\K}(V)[x_1,\ldots,x_n]$ and non-empty Zariski-open sets $U^i\subset \tilde{U}$, such that $\pi_i^{-1}\pi_i(T(g(I)))$ is either empty or a tropical hypersurface generated by $F_i(g)\in g(I)$, such that $T(F_i(g))$ is the same polyhedral complex for all $g\in U^i$ by Theorem \ref{tropbasiselement}. In particular, for a fixed index $i$ the set $\pi_i^{-1}\pi_i(T(g(I)))$ is the same subset of $\R^n$ for all $g\in U^i$. Let $U=\bigcap_{i=0}^n U^i$, which as an intersection of
finitely many non-empty open sets is itself open. As $$\bigcap_{i=0}^n \pi_i^{-1} \pi_i(T(g(I)))=T(g(I))$$ is the same set for all $g\in U$ as well, this proves the existence of generic tropical varieties as a set.

Assume that $T(g(I))\neq\emptyset$. Since the tropical variety $T(g(I))$ is a subcomplex of the generic Gr\"obner complex $\gGC(I)$ for every $g\in U\subset \tilde{U}$, we have a natural complex structure on $T(g(I))$. It follows that $T(g(I))$ is also constant as a polyhedral complex for all $g\in U$ with this complex structure induced by Gr\"obner basis theory.

Moreover, one can obtain a generic tropical basis of $I$ as follows. Since the $F_i(g)$ already cut out the tropical variety for $g\in U$, we only need to add a finite generating system of constant support. For this choose homogeneous generators $f_1,\ldots,f_s$ of $I$. Then $y(f_1),\ldots,y(f_s)$ generate $y(I)\subset {\K}(V)[x_1,\ldots,x_n]$. By Lemma \ref{trophypeconst} we can choose $\emptyset\neq U'\subset U$ open, such that $T(g(f_i))$ is the same polyhedral complex for all $g\in U'$ and all $i$. Adding the $y(f_1),\ldots,y(f_s)\in y(I)$ to the set of the $F_i(y)\in y(I)$ yields a generic tropical basis of $I$ on $V$ valid on $U'$.
\end{proof}

With Theorem \ref{tropbasics} the assumption that $I$ is prime can be dropped. We need the following auxiliary result.

\begin{lemma}\label{monlift}
Let $V\subset \GL_n({\kx})$ a subvariety, $P\subset S_{\K}$ be a graded prime ideal and $y(P)$ be its extension in $S_{{\K}(V)}$ via the inclusion given by Definition \ref{gendef}. If $y(P)$ contains no monomial, then there exists $g\in V$, such that $g(P)$ contains no monomial.
\end{lemma}

\begin{proof}
Although the valuation on ${\K}$ is not trivial, in this proof we use classical Gr\"obner basis theory in $S_{{\K}(V)}$ and $S_{\K}$. Choose any term order $\succ$ on $S_{{\K}(V)}$ (this is a term order on $S_{\K}$ as well). By Theorem \ref{gengroebfanonsubvar} there exists a non-empty open subset $U\subset V$, such that the reduced Gr\"obner basis $\Gc(y)=\left\{f_1(y),\ldots,f_s(y)\right\}$ of $y(P)$ with respect to $\succ$ and the sets $\Gc(g)$, where the $g_{ij}$ are substituted for $y_{ij}$, have the same support for all $g\in U$. Moreover, $\Gc(g)$ is the reduced Gr\"obner basis of $g(P)$ with respect to $\succ$ for $g\in U$. Assume that $y(P)$ does not contain a monomial. In particular, we then have $x_1\cdots x_n\notin y(P)$. Dividing $x_1\cdots x_n$ by the Gr\"obner basis $\Gc(y)$ (in the sense of \cite[Algorithm 1.3.4]{MATH}) yields an expression $$x_1\cdots x_n=\sum_{j=1}^s f_j(y)r_j(y)+r(y),$$ where $f_j(y)\in \Gc(y)$, $r_j(y),r(y)\in S_{{\K}(V)}$ and $r(y)$ is the normal form of $x_1\cdots x_n$ with respect to $\Gc(y)$. Since $x_1\cdots x_n\notin y(P)$, the polynomial $r(y)\neq 0$. Thus there exists an open subset $\emptyset\neq U'\subset U$, such that $r(g)\neq 0$ in $S_{\K}$ for $g\in U'$. As $\Gc(g)=\left\{f_1(g),\ldots,f_s(g)\right\}$ is a Gr\"obner basis of $g(P)$ and $r(g)$ is the normal form of $x_1\cdots x_n$ with respect to $\Gc(g)$, this implies $x_1\cdots x_n\notin g(P)$ for $g\in U'$. Hence, $g(P)$ cannot contain a monomial for $g\in U'$, since it is prime and every prime ideal in $S_{\K}$ containing a monomial contains the particular monomial $x_1\cdots x_n$. This proves the claim.
\end{proof}

\begin{theorem}\label{thm:main}
Let $I\subset S_{\K}$ be a graded ideal and $V\subset \GL_n({\kx})$ a subvariety. Then $\gT_V(I)$ exists. If $\gT_V(I)\neq\emptyset$, there exists a generic tropical basis of $I$ on $V$.
\end{theorem}

\begin{proof}
Let $P_1,\ldots, P_t$ be the minimal prime ideals of $I$. By Theorem \ref{thm:mainprime} there
are Zariski-open sets $\emptyset\neq U_i\subset V$, such that
$T(g(P_i))$ is either empty or the same tropical variety for every $g\in U_i$. Since $g(P_1),\ldots,g(P_t)$ are the minimal prime
ideals of $g(I)$, by Theorem \ref{tropbasics} this implies $$T(g(I))=\bigcup_{i=1}^t T(g(P_i))$$
is the same set for every $g\in U=\bigcap_{i=1}^t U_i$. Analogously to the end of the
proof of Theorem \ref{thm:mainprime} one can additionally conclude
that $T(g(I))$ is also constant as a polyhedral complex for all $g\in U$ if it is non-empty.

Moreover, if $\gT_V(I)\neq\emptyset$, we can also obtain a generic tropical basis of $I$. The idea is to add a finite generating system of $y(I)$ (analogously to the proof of Theorem \ref{thm:mainprime}) to a set of polynomials in $y(I)$ which cut out $\gT_V(I)$ as follows:

We proceed by induction of the number $t$ of minimal primes of $I$. If $t=1$ and $P_1$ is the only minimal prime of $I$, then by Remark \ref{primesinext} the ideal $y(I)$ has only one minimal prime, which is $y(P_1)$. Thus $\sqrt{y(I)}=y(P_1)$. Let $\left\{F_1(y),\ldots,F_l(y)\right\}$ be a generic tropical basis of $P_1$ valid on an open subset $\tilde{U}\subset V$, which exists by Theorem \ref{thm:mainprime}. Since $F_i(y)\in y(P_1)=\sqrt{y(I)}$, there exists $n_i\in \N$, such that $F_i(y)^{n_i}\in y(I)$. This implies $F_i(g)^{n_i}\in g(I)$ for $g\in \tilde{U}$. Moreover, $T(F_i(g)^{n_i})=T(F_i(g))$ by Theorem \ref{tropbasics}, so the set $\left\{F_1(g)^{n_1},\ldots,F_l(g)^{n_l}\right\}$ also cuts out the tropical variety $T(g(I))$. Analogously to the procedure in the proof of Theorem \ref{thm:mainprime} we can add $\left\{F_1(y)^{n_1},\ldots,F_l(y)^{n_l}\right\}$ to a finite generating system of $y(I)$ to obtain a generic tropical basis of $I$.

Let $t>1$ and $P_1,\ldots,P_t$ be the minimal primes of $I$. If $\gT_V(I)\neq \emptyset$, we can assume without loss of generality that $\gT_V(P_1)\neq \emptyset$. By induction hypothesis we then have a generic tropical basis $\left\{H_1(y),\ldots,H_s(y)\right\}$ of $I':=\bigcap_{i=1}^{t-1} P_i$. We have to consider two cases:

\begin{enumerate}
\item If $\gT_V(P_t)\neq\emptyset$, we also have a generic tropical basis $\left\{F_1(y),\ldots,F_l(y)\right\}$ of $P_t$. Let $\emptyset\neq U'\subset U$ be open, such that both tropical bases are valid on $U'$. For $g\in U'$ we have
\begin{eqnarray*}
T(g(I)) & = & T\left(g(I')\cap g(P_t)\right)\\
        & = & T\left(g(I')\right)\cup T\left(g(P_t)\right)\\
        & = & \left(\bigcap_{j=1}^s T(H_j(g))\right)\cup \left(\bigcap_{k=1}^l T(F_k(g))\right)\\
        & = & \bigcap_{j,k} \left(T(H_j(g))\cup T(F_k(g))\right)\\
        & = & \bigcap_{j,k} T\left(H_jF_k(g)\right).
\end{eqnarray*}
So the products $H_jF_k(g)$ cut out $T(g(I))$ for $g\in U'$. All products $H_jF_k(y)$ are elements of $$\left(\bigcap_{i=1}^{t-1} y(P_i)\right)y(P_t)\subset \bigcap_{i=1}^{t} y(P_i)=\sqrt{y(I)},$$ see Remark \ref{primesinext}. This implies that we can choose $n_{jk}\in \N$, such that $(H_jF_k(y))^{n_{jk}}\in y(I)$. Since we know $T(H_jF_k(g))=T((H_jF_k(g))^{n_{jk}})$ for any $g$, we can add the $H_jF_k(y)^{n_{jk}}$ to a finite generating set of $y(I)$. This yields a generic tropical basis of $I$ on $V$.

\item If $\gT(P_t)=\emptyset$, we know that $g(P_t)$ contains a monomial for all $g\in V$ by Theorem \ref{tropbasiselement}. By Lemma \ref{monlift} this implies that $y(P_t)\subset S_{{\K}(V)}$ contains a monomial $F$. Since multiplying by a monomial does not change a tropical hypersurface, we have 
\begin{eqnarray*}
T(g(I)) & = & T(g(I'))\cup T(g(P_t))\\
        & = & T(g(I'))\\
        & = & \bigcap_{j=1}^s T(H_j(g))\\
        & = & \bigcap_{j=1}^s T(H_jF(g)).
\end{eqnarray*}
Since $F$ is an element of $y(P_t)$, it follows that $H_kF(y)\in y(I')\cap y(P_t)=\sqrt{y(I)}$. Choose integers $n_1,\ldots,n_s\in \N$, such that $H_kF(y)^{n_k}\in y(I)$. Adding $$\left\{H_1F(y)^{n_1},\ldots,H_sF(y)^{n_s}\right\}$$ to a generating set of $y(I)$ as constructed above yields a generic tropical basis of $I$ on $V$.
\end{enumerate}
\end{proof}

We end this chapter with two basic observations about the structure, that the different possible tropical varieties of an ideal $I$ induce on $\GL_n({\kx})$ in the following sense. By considering $g,g'\in \GL_n({\kx})$ to be equivalent if $T(g(I))=T(g'(I))$, it is a natural question, of what can be said about the corresponding equivalence classes. While it is hard in general to give a complete description of this structure, some fundamental properties can be established directly. We first show the set of all coordinate transformations which induce an empty tropical variety to be a closed subset of $\GL_n({\kx})$. This is proved by repeated use of Theorem \ref{tropbasiselement}.

\begin{cor}
Let $I\subset S_{\K}$ be graded ideal. Then the set $\left\{g\in\GL_n({\kx}): T(g(I))=\emptyset\right\}$ is closed in $\GL_n({\kx})$.
\end{cor}

\begin{proof}
We first consider the case, that $I$ is a prime ideal. Denote $$\left\{g\in\GL_n({\kx}): T(g(I))=\emptyset\right\}$$ by $M_I$. We proceed by inductively applying the following fact, which holds by Theorem \ref{tropbasiselement}: If $W\subset \GL_n({\kx})$ is an irreducible subvariety, either $W\subset M_I$ or $W\cap M_I\subsetneq W$ is contained in a closed proper subset of $W$.

To start, $\GL_n({\kx})$ is an irreducible subvariety of itself, so we either have $M_I=\GL_n({\kx})$ or $M_I\subsetneq \GL_n({\kx})$ is contained in a closed proper algebraic subset $W^1\subset \GL_n({\kx})$. In the first case there is nothing to prove, while in the second case note that $\dim W^1\leq n^2-1$. Let $W^1_1,\ldots,W^1_{t_1}$ be the irreducible components of $W^1$. By the above statement we can assume that $W^1_1,\ldots,W^1_{s_1}$ are the irreducible components with $W^1_k\subset M_I$ and $W^1_{s_1+1},\ldots,W^1_{t_1}$ are the irreducible components of $W^1$ with $M_I\cap W^1_k\subsetneq W^1_k$ is contained in a proper closed algebraic subset $W^2_k$ of $W^1_k$. Then $\dim W^2_k\leq n^2-2$. This yields the chain of inclusions $$\bigcup_{i=1}^{s_1} W^1_i\subset M_I\subset (\bigcup_{i=1}^{s_1} W^1_i)\cup (\bigcup_{i=s_1+1}^{t_1} W^2_i).$$

To proceed consider the irreducible components of $W^2=\bigcup_{i=s_1+1}^{t_1} W^2_i$ and apply Theorem \ref{tropbasiselement} to obtain a sequence of inclusions $$(\bigcup_{i=1}^{s_1} W^1_i)\cup (\bigcup_{i=1}^{s_2} W^2_i)\subset M_I \subset (\bigcup_{i=1}^{s_1} W^1_i)\cup (\bigcup_{i=1}^{s_2} W^2_i) \cup (\bigcup_{i=s_2+1}^{t_2} W^3_i),$$ where $W^2_1,\ldots,W^2_{s_2}$ are the irreducible components of $W^2$ contained $M_I$ and $W^2_{s_2+1},\ldots,W^2_{t_2}$ are the irreducible components of $W^2$ with $M_I\cap W^2_k\subset W^3_k$ for $W^3_k\subset W^2_k$ closed with $\dim W^3_k\leq n^2-3$ for each such $k$. By inductively decreasing the dimension of the closed set marking the difference between the sets on the left and right side of the chain of inclusions, we obtain the desired result for the case that $I$ is prime.

Let $I\subset S_{\K}$ be an arbitrary graded ideal and $P_1,\ldots,P_q$ its minimal primes. Then by Theorem \ref{tropbasics} we know that $T(g(I))=\bigcup_{r=1}^{q} T(g(P_r))$ for $g\in \GL_n({\kx})$. So for a given $g$ we have $T(g(I))=\emptyset$ if and only if $T(g(P_r))=\emptyset$ for every $r$. Hence, $M_I=\bigcap_{r=1}^q M_{P_r}$, which is itself closed as an intersection of the closed subsets $M_{P_r}$.
\end{proof}

Moreover, we can show that the set of equivalence classes of the above equivalence relation is finite, giving rise to only finitely many possible tropical varieties of a fixed ideal under an arbitrary (linear) coordinate change.

\begin{cor}\label{endlmoegl}
Let $I\subset S_{\K}$ be a graded ideal. Then there exist finitely many polyhedral complexes $\Fc_1,\ldots,\Fc_t$ in $\R^n$, such that for any $g\in \GL_n({\kx})$ we $T(g(I))=\Fc_k$ for some $k$.
\end{cor}

Note that one of the $\Fc_k$ in the statement can be the empty polyhedral complex.

\begin{proof}
We proceed by inductively cutting out Zariski-open sets of subvarieties of $\GL_n({\kx})$ for which $T(g(I))$ is the same polyhedral complex by using Theorem \ref{thm:main} repeatedly. In each step of this process the dimension of the remaining set strictly decreases, as the complement of a non-empty open set has always a strictly smaller dimension. For the first step let $\emptyset\neq U\subset\GL_n({\kx})$ be open, such that $T(g(I))=\gT(I)$ for every $g\in U$, which exists by Theorem \ref{thm:main}. The complement of $U\subset \GL_n({\kx})$ is closed and, hence, is the union of finitely many irreducible subvarieties of dimension less than $n^2$. Again by Theorem \ref{thm:main} we choose a non-empty Zariski-open subset $U_V$ of each such component $V$, such that $T(g(I))$ is the same (possibly empty) polyhedral complex for every $g\in U_V$. The complement of $U\cup(\bigcup_V U_V)$ in $\GL_n({\kx})$ now has dimension less than $n^2-1$ and thus is the union of finitely many subvarieties of dimension less than $n^2-1$. Proceeding inductively we add a finite number of possible polyhedral complexes for $T(g(I))$ to our collection in each step, while decreasing the dimension of the set of the remaining $g\in\GL_n({\kx})$. This algorithm stops when the remaining set has dimension $0$, i.e. it is a union of finitely many points. We can finally add the tropical varieties corresponding to those points to our collection of polyhedral complexes, thereby obtaining the desired result.
\end{proof}

This statement of course raises the natural question of classifying all possible tropical varieties of a given ideal or class of ideals in cases of interest.

\section{Examples}\label{examples}

We conclude this paper with three classes of examples where generic Gr\"obner complexes and generic tropical varieties are directly computable sketching the ideas and referring to \cite[Chapter 4]{ICH} for full proofs.

\begin{ex}
In \cite[Theorem 4.5]{TK} it was shown that in the constant coefficient case the generic tropical variety of a graded ideal $I\subset S_K$ on $\GL_n({\kx})$ as a set depends only on the dimension of $S_K/I$. This seems to offer a rather coarse distinction of ideals by their tropical varieties. It raises the question whether one can make a finer differentiation by choosing a suitable subvariety $V\subset \GL_n({\kx})$, for which $\gT_V(I)$ can be different for ideals of the same dimension, but is still computable. One such subvariety is the group of diagonal matrices in $\GL_n({\kx})$. While generic tropical varieties over $\GL_n({\kx})$ provide a very rough distinction of ideals, we will show that generic tropical varieties over diagonal matrices constitute an example for the other extreme.

Let $D_n=\left\{g\in \GL_n({\kx}): g_{ij}=0 \text{ for } i\neq j\right\}$ be the set of all diagonal matrices of $\GL_n({\kx})$. This set is as well a subgroup of $\GL_n({\kx})$ as an $n$-dimensional subvariety. In tropical geometry it plays a role in the study of singularities of tropical curves, see \cite{MAMASH}. We do not assume to be in the constant coefficient case in this section.

By Theorem \ref{thm:main} we know that for every graded ideal $I\subset S_{\K}$ there exists a non-empty Zariski-open set $U\subset D_n({\kx})$, such that $T(g(I))$ is the same polyhedral complex for every $g\in U$, i.e. the generic tropical variety $\gT_{D_n({\kx})}(I)$ on $D_n({\kx})$ exists. It is thus a natural problem to describe $\gT_{D_n({\kx})}(I)$ and to ask what information of $I$ can be obtained from it. One can show directly that $T(g(I))=T(I)$ for all $g\in D_n$. The main reason for this is, that in this special case taking initial forms commutes with changing coordinates: For a homogeneous polynomial $f\in S_K$, $g\in D_n$ and $\omega\in \R^n$ we have $\inom_{\omega}(g(f))=g(\inom_{\omega}(f))$. From this it follows that $\inom_{\omega}(g(I))=\inom_{\omega'}(g(I))$ if and only if $\inom_{\omega}(I)=\inom_{\omega'}(I)$ for $\omega,\omega'\in\R^n$ and $g\in D_n$. In particular, $I$ and $g(I)$ have the same Gr\"obner complex. Showing that the same polyhedra in this Gr\"obner complex correspond to monomial-free initial ideals for $I$ and $g(I)$ yields that in fact $T(I)=T(g(I))$ for $g\in D_n$. This shows that for a graded ideal $I\subset S_L$ there is always the $n$-dimensional subvariety $D_n({\kx})$ of $\GL_n({\kx})$, such that $T(g(I))=T(I)$ for every $g\in D_n({\kx})$. 
\end{ex}

\begin{ex}
For principal ideals generated by some homogeneous polynomial $0\neq f\in S_L={\K}[x_1,\ldots,x_n]$ we can explicitly describe the generic Gr\"obner complex and the generic tropical variety on $\GL_n({\kx})$. Even in the non-constant coefficient case $\gGC(f)$ and $\gT(f)$ are fans, which are closely related to the generic tropical fan, see for example \cite[Definition 4.1]{TK}. Analogously to \cite[Lemma 5.1]{TK} we can find an open subset $U\subset \GL_n(K)$ such that
\begin{enumerate}
\item $g(f)$ has constant support and every monomial $x_j^{d}$ appears with non-zero coefficient,
\item the valuations of the coefficients of all monomials appearing in $g(f)$ are independent of $g$,
\item the valuation of the coefficients of the $x_j^{d}$ are all the same and minimal among the valuations of all coefficients
\end{enumerate}
for $g\in U$. This is done by considering the polynomial $y(f)\in S_{L(\GL_n(K))}$, where $L(\GL_n(K))$ is the field as in Notation \ref{kv}. By Corollary \ref{valfieldext} the field $L(\GL_n(K))$ is a valued subfield of $K(\GL_n(K))\left\{\!\!\left\{t\right\}\!\!\right\}$. Regarding $y(f)$ as an element of $S_{K(\GL_n(K))\left\{\!\!\left\{t\right\}\!\!\right\}}$ we can choose $U$ such that none of the coefficients of $y(f)$ and none of the leading terms of those coefficients vanish when substituting $g$ for $y$ for any $g\in U$. This set $U$ fulfills the above conditions.

For a homogeneous polynomial $0\neq f \in S_L$ this yields a complete characterization of the generic Gr\"obner fan and generic tropical variety on $\GL_n(K)$. In fact:
\begin{enumerate}
\item $\gGC(f)$ is equal to the generic tropical fan $\Wc_n$.
\item $\gT(f)$ is equal to $\Wc_n^{n-1}$, the $(n-1)$-skeleton of the generic tropical fan.
\end{enumerate}
This is done by the same reasoning as in the proof of \cite[Proposition 5.2]{TK}
\end{ex}

\begin{ex}
Tropical varieties of linear ideals, i.e. ideals generated by linear forms, in the constant coefficient case have been studied in \cite{ARKL} using the theory of matroids. If we choose $V=\GL_n({\kx})$, the generic Gr\"obner complex and generic tropical variety of a linear ideal can be computed explicitly. Even in the non-constant coefficient case both objects are fans in $\R^n$. Moreover, these fans just depend on the dimension of the ideal and have a very symmetric structure.

One feature of linear ideals making this class accessible to Gr\"obner basis computations is the direct connection of their structure to basic linear algebra. To be able to apply linear algebra methods the first claim is that for $\omega\in\R^n$ the ideal $\inom_{\omega}(I)$ is linear. By Proposition \ref{hilbinit} the Hilbert functions of $(S_{\K}/I)$ and $(S_{\kx}/\inom_{\omega}(I))$ agree. Hence, the multiplicities of the corresponding projective varieties are the same. Since $I$ is linear, this multiplicity is $1$. As linear varieties are exactly the varieties of multiplicity $1$ (see \cite[Exercise I,7.6.]{H}), this means that $\inom_{\omega}(I)$ is linear. Hence, by Theorem \ref{gengrobcomp} the generic Gr\"obner complex $\gGC(I)$ is equal to the generic Gr\"obner complex $\Cc^1$ in degree $1$ as introduced in Lemma \ref{degcomplex}.

To show that $\Cc^1$ is a fan note that for $g\in \GL_n(K)$ the ${\K}$-vector space $g(I)_1$ corresponds to a point in the Grassmannian $\Gr_{\K}(n-m,n)$. This holds, as there are $n$ monomials of degree $1$ in $S_L$ and we have $\dim_{\K} g(I)_1=n-m$. Thus the Pl\"ucker coordinates $P_J(g)$ of $g(I)_1$ are indexed by sets $J\subset \left\{1,\ldots,n\right\}$ of cardinality $n-m$. 

By the proof of Lemma \ref{groebpoly} every Gr\"obner polyhedron $C_{g(I)}^1[\omega]$ of $g(I)$ in degree $1$ is defined by equalities and inequalities among the expressions $$v(P_J(g))+\omega\cdot M_J,$$ where $M_J=\sum_{j\in J}\mu_j$ is the sum over the exponents of all monomials indexed by $J$ as in the proof of Lemma \ref{groebpoly}. The key point now is to show that there exists $\emptyset \neq U \subset \GL_n(K)$ such that for every $g\in $ we have $v(P_J(g))=v(P_{J'}(g))$ for every $J,J'\in [n]^{(n-m)}$. We then have equalities and inequalities of the form $\omega\cdot(M_J-M_{J'})<0$ and $\omega\cdot (M_J-M_{J'})=0$, which define a cone in $\R^n$. To prove the claim note that we can obtain an open set $\emptyset\neq U'\subset \GL_n(K)$ such that for given $J$ the real number $v(P_J(g))$ is the same for every $g\in U'$ by Lemma \ref{groebpoly}. Moreover, for every $J,J'\in [n]^{(n-m)}$ we can choose a coordinate permutation $\tau\in \GL_n(K)$ such that $P_J(\tau\circ g)=P_{J'}(g)$ (this is the step where the proof fails for general non-linear ideals, as this claim is false there). As every non-empty open set contains a non-empty open set $W$ such that $g\in W$ implies $\tau(g)\in W$ for every permutation $\tau$, the claim follows by considering such a set $W\subset U'$.

Comparing the inequalities above then yields the following result. Let $I\subset S_{\K}$ be a linear ideal with $\dim (S_{\K}/I)=m$. Then the generic Gr\"obner complex $\gGC(I)$ contains the following maximal cones: For $\omega\in \R^n$, such that $$\omega_{i_1},\ldots,\omega_{i_{n-m}}<\omega_{i_{n-m+1}},\ldots,\omega_{i_n}$$ with $\left\{i_1,\ldots,i_n\right\}=\left\{1,\ldots,n\right\}$ we have $$C[\omega]=\left\{\omega'\in\R^n: \omega'_{i_1},\ldots,\omega'_{i_{n-m}}<\omega'_{i_{n-m+1}},\ldots,\omega'_{i_n}\right\}.$$ 

For the generic tropical variety $\gT(I)$ we now claim that $\gT(I)=\Wc_n^m$, the $m$-skeleton of the generic tropical fan (even in the non-constant coefficient case). As $\gT(I)$ is a subcomplex of $\gGC(I)$, we only have to show that $\left|\gT(I)\right|=\left|\Wc_n^m\right|$ by the above result. This can be proved along the following lines: Let $\omega\in \R^n$ and $g\in W$ as above. By definition $\omega\notin \gT(I)$ if and only if $\inom_{\omega}(g(I))$ contains a monomial. As $\inom_{\omega}(g(I))$ is linear, this is equivalent to $\inom_{\omega}(g(I))$ containing a variable, say $x_k$. This is true if and only if $\inom_{\omega}(g(I))_1$ as a vector subspace of $L^n$ contains the standard basis vector $e_k$ for that $k$. Comparing the equalities and inequalities above this is the same as saying that $\omega_k<\omega_j$ for at least $m$ indices $j\neq k$. This statement is equivalent to the fact that $\min_j\left\{\omega_j\right\}$ is attained at most $n-m$ times, so $\omega\notin \Wc_n^m$ completing the sketch of proof.
\end{ex}


\begin{thebibliography}{99}



\bibitem{ARKL}
F.\ Ardila and C.\ Klivans, \emph{The Bergman complex of a matroid and phylogenetic trees}.
J. Comb. Theory, Ser. B {\bf 96}, No. 1, 38-49 (2006). MSC2000.



\bibitem{BIGR}
R.\ Bieri and J.R.J.\ Groves, \emph{The geometry of the set of
characters induced by valuations}. J. Reine Angew. Math. {\bf 347},
168--195 (1984).


\bibitem{BOJESPSTTH}
T.\ Bogart, A.N.\ Jensen, D.\ Speyer, B.\ Sturmfels and R.R.\ Thomas, \emph{Computing tropical varieties}. J. Symb. Comput. {\bf
42}, No. 1-2, 54--73 (2007).






\bibitem{DIFEST}
A.\ Dickenstein, E.\ M.\ Feichtner and B.\ Sturmfels, \emph{Tropical discriminants}. Journal of the American Mathematical Society {\bf 20}, 1111--1133 (2007).

\bibitem{DR}
J.\ Draisma, \emph{A tropical approach to secant dimensions}. J. Pure Appl. Algebra {\bf 212}, No. 2, 349--363 (2008).

\bibitem{E}
D.\ Eisenbud, \emph{Commutative algebra. With a view toward algebraic geometry.}
Graduate Texts in Mathematics {\bf 150}. Springer (1995).



\bibitem{GA}
A.\ Gathmann, \emph{Tropical algebraic geometry}. Jahresber. Dtsch. Math.-Ver. {\bf 108}, No. 1, 3--32 (2006).





\bibitem{GRE}
M.L.\ Green, \emph{Generic initial ideals}. In: J.\ Elias (ed.) et al., Six lectures on commutative algebra.
Birkhäuser, Prog. Math. {\bf 166}, 119--186 (1998).


\bibitem{H}
R.\ Hartshorne, \emph{Algebraic Geoemtry}. 
Graduate Texts in Mathematics {\bf 52}. Springer (1977).

\bibitem{HEP}
K.\ Hept, \emph{Projections of Tropical Varieties and an Application to Small Tropical Bases}. Dissertation, Goethe-Universit\"at Frankfurt (2009).

\bibitem{HETH}
K.\ Hept and T.\ Theobald, \emph{Tropical bases by regular
projections}. Proc. Amer. Math. Soc. {\bf 137}(7), 2233--2241 (2009).








\bibitem{MA}
D.\ Maclagan, \emph{Antichains of monomial ideals are finite}. Proc. Amer. Math. Soc., {\bf 129}(6): 1609-1615 (electronic), (2001).

\bibitem{MAST}
D.\ Maclagan and B.\ Sturmfels, \emph{Introduction to Tropical Geometry}. Book in preparation. http://www.warwick.ac.uk/staff/D.Maclagan/papers/TropicalBook.pdf

\bibitem{MATH}
D.\ Maclagan, R.\ Thomas,  S.\ Faridi, L.\ Gold, A.V.\ Jayanthan, A. Khetan and T.\ Puthenpurakal,
\emph{Computational Algebra and
Combinatorics of Toric Ideals}. In: R.V.\ Gurjar (ed.) et al.,
Commutative algebra and combinatorics. Ramanujan Mathematical
Society Lecture Notes Series {\bf 4}, Ramanujan Mathematical Society (2007).

\bibitem{MAMASH}
H.\ Markwig, T.\ Markwig and E.\ Shustin,
\emph{Tropical curves with a singularity in a fixed point}.
Preprint, arXiv:0909.1827.

\bibitem{MAR}
T.\ Markwig, \emph{A Field of Generalised Puiseux Series for Tropical Geometry}. Rend. Semin. Mat. Torino {\bf 68}, 1, 79--92 (2010).

\bibitem{MI1}
G.\ Mikhalkin. \emph{Enumerative tropical algebraic geometry in $\R^2$}. J. Am. Math. Soc. {\bf 18}, No. 2, 313--377 (2005).


\bibitem{MORO}
T.\ Mora and L.\ Robbiano, \emph{The Gr\"obner fan of an ideal}. J. Symb. Comput. {\bf 6}, No.2--3, 183--208 (1988).

\bibitem{PA}
S.\ Payne, \emph{Fibers of tropicalization}. Math. Z. {\bf 262}, 301-311 (2009).



\bibitem{RI}
P.\ Ribenboim, \emph{Fields: Algebraically Closed and Others}. Manuscripta Math. {\bf 75}, No. 2, 115--150 (1992).

\bibitem{TK}
T.\ R\"omer and K.\ Schmitz, \emph{Generic Tropical Varieties}, Journal of Pure and Applied Algebra (2011), doi:10.1016/j.jpaa.2011.05.012

\bibitem{TK2}
T.\ R\"omer and K.\ Schmitz, \emph{Algebraic Properties of Generic Tropical Varieties}. 
Algebra Number Theory {\bf 4}, No. 4, 465-491 (2010).

\bibitem{ICH}
K.\ Schmitz, \emph{Generic Tropical Varieties}, Dissertation, Osnabr\"uck (2011).






\bibitem{STTE}
B.\ Sturmfels and J.\ Tevelev, \emph{Elimination theory for tropical varieties}. Math. Res. Lett. {\bf 15}, No. 2--3, 562--543 (2008).




\end{thebibliography}
\end{document}